\documentclass[11pt]{article}
\usepackage{amsmath,amssymb,amsthm}
\usepackage[usenames,dvipsnames]{color}
\usepackage{enumerate}
\usepackage{geometry}
\usepackage{graphicx}
\usepackage{fullpage}

\numberwithin{equation}{section}
\newtheorem{theorem}[equation]{Theorem}

\newtheorem{lemma}[equation]{Lemma}
\theoremstyle{definition}
\newtheorem{remark}[equation]{Remark}

\newtheorem{algorithm}{Algorithm}[section]

\newcommand{\dt}{{\Delta t}}
\newcommand{\R}{\mathbb{R}}

\usepackage[backgroundcolor=gray!30,linecolor=black]{todonotes}

\title{Global in time stability and accuracy of IMEX-FEM data assimilation schemes for the Navier-Stokes equations}

\author{
Adam Larios
\thanks{Department of Mathematics, University of Nebraska-Lincoln, Lincoln, NE 68588; 
email: alarios@unl.edu, partially supported by NSF Grant DMS 1716801}
\and
Leo G. Rebholz\thanks{Department of Mathematical Sciences, Clemson University, Clemson, SC, 29634;
email: rebholz@clemson.edu, partially supported by NSF Grant DMS 1522191.}
\and
Camille Zerfas
\thanks{Department of Mathematical Sciences, Clemson University, Clemson, SC, 29634;
email: czerfas@clemson.edu, partially supported by NSF Grant DMS 1522191.}
}

\begin{document}
\date{}
\maketitle

\begin{abstract}
We study numerical schemes for incompressible Navier-Stokes equations using IMEX temporal discretizations, finite element spacial discretizations, and 
equipped with continuous data assimilation (a technique recently developed by Azouani, Olson, and Titi in 2014).   We analyze stability 
and accuracy of the proposed methods, and are able to prove well-posedness, long time stability, and long time accuracy estimates, 
under restrictions of the time step size and data assimilation parameter.  We give results for several numerical tests that illustrate the theory,
and show that, for good results, the choice of discretization parameter and element choices can be critical.
\end{abstract}

\section{Introduction}

Data assimilation (DA) refers to a wide class of schemes for incorporating observational data in simulations, in order to increase the accuracy of solutions and to obtain better estimates of initial conditions.  It is the subject of a large body of work (see, e.g., \cite{Daley_1993_atmospheric_book,Kalnay_2003_DA_book,Law_Stuart_Zygalakis_2015_book}, and the references therein).  DA algorithms are widely used in weather modeling, climate science, and hydrological and environmental forecasting \cite{Kalnay_2003_DA_book}.  Classically, these techniques are based on {\color{black} linear quadratic} estimation, also known as the Kalman Filter.  The Kalman Filter is described in detail in several textbooks, including \cite{Daley_1993_atmospheric_book,Kalnay_2003_DA_book,Law_Stuart_Zygalakis_2015_book,CHJ69}, and the references therein.  

Recently, a promising new approach to data assimilation was pioneered by Azouani, Olson, and Titi \cite{Azouani_Olson_Titi_2014,Azouani_Titi_2014} (see also \cite{Cao_Kevrekidis_Titi_2001,Hayden_Olson_Titi_2011,Olson_Titi_2003} for early ideas in this direction).  This new approach, which we call AOT Data Assimilation or continuous data assimilation, adds a feedback control term at the PDE level that nudges the computed solution towards the reference solution corresponding to the observed data. A similar approach is taken by Bl\"omker, Law, Stuart, and Zygalakis in \cite{Blomker_Law_Stuart_Zygalakis_2013_NL} in the context of stochastic differential equations.  The AOT algorithm is based on feedback control at the PDE (partial differential equation) level, described below.  The first works in this area assumed noise-free observations, but \cite{Bessaih_Olson_Titi_2015} adapted the method to the case of noisy data, and \cite{Foias_Mondaini_Titi_2016} adapted to the case in which  measurements are obtained discretely in time and may be contaminated by systematic errors.   Computational experiments on the AOT algorithm and its variants were carried out in the cases of the Navier-Stokes equations \cite{Gesho_Olson_Titi_2015,Leoni_Mazzino_Biferale_2018}, the B\'enard convection equations \cite{Altaf_Titi_Knio_Zhao_Mc_Cabe_Hoteit_2015}, and the Kuramoto-Sivashinsky equations \cite{Lunasin_Titi_2015,Larios_Pei_2017_KSE_DA_NL}.  In \cite{Larios_Pei_2017_KSE_DA_NL}, several nonlinear versions of this approach were proposed and studied. In addition to the results discussed here, a large amount of recent literature has built upon this idea; see, e.g., \cite{Albanez_Nussenzveig_Lopes_Titi_2016,Biswas_Martinez_2017,Farhat_Jolly_Titi_2015,Farhat_Lunasin_Titi_2016abridged,Farhat_Lunasin_Titi_2016benard,Farhat_Lunasin_Titi_2016_Charney,Farhat_Lunasin_Titi_2017_Horizontal,Foyash_Dzholli_Kravchenko_Titi_2014,GlattHoltz_Kukavica_Vicol_2014,Jolly_Martinez_Titi_2017,Jolly_Sadigov_Titi_2015,Markowich_Titi_Trabelsi_2016}.
Although extensive research has been done on the theory of DA algorithms, there are far fewer papers on the numerical analysis of these algorithms.  We note that a continuous-in-time Galerkin approximation of the algorithm was studied in \cite{Mondaini_Titi_2018_SIAM_NA}.  Also, recently \cite{IMT18} studied a Galerkin in space algorithm with semi-implicit and implicit time-stepping for the 2D Navier-Stokes equations (NSE) which are first-order in time (i.e., Euler methods).

In this paper, we propose and study discrete numerical algorithms of the 3D NSE with an added data assimilation term and grad-div term, and under the assumption that sufficiently regular solutions exist.  In particular, we consider \textit{second order} implicit/explicit (IMEX) time stepping schemes and finite element spacial discretizations.  The semi-implicit scheme we propose and analyze for the 3D NSE (Algorithm \ref{alg1a} below) studied is similar to the algorithm in \cite{IMT18} for the 2D NSE, with one difference being our use of the grad-div stabilization.  The analysis also differs due to the change in dimension.  As far as we are aware, the present work contains the first proposed higher-order time-stepping scheme for the AOT algorithm, and the first numerical analysis of an AOT scheme for the 3D NSE.  In addition, we show that the particular element choice and/or stabilization parameters can make a dramatic difference in the success of the DA algorithm, and the time stepping algorithms also need careful consideration since time step restrictions can arise.  We also show some computational tests of our algorithms in the 2D case in several benchmark settings.  This includes what we believe are the first computational tests of the algorithm for capturing lift and drag in the setting of 2D channel flow past a cylindar, as well as results that show that AOT data assimilation can fail if standard element choices are made, but can work quite will with divergence-free finite elements

Briefly, the incompressible NSE are given by 
\begin{align}
u_t + (u\cdot \nabla )u - \nu \Delta u + \nabla p - \gamma \nabla (\nabla \cdot u) &= f, \label{nse1}
\\ \nabla \cdot u & = 0, \label{nse2} 
\end{align}
where $u$ represents the velocity and $p$ pressure. The viscosity is given by $\nu>0$, and external forcing is $f$. We include a grad-div stabilization term with parameter $\gamma>0$. Note that at the continuous level, this term is zero. The corresponding data assimilation algorithm is given by the system, 
\begin{align}
v_t + (v\cdot \nabla )v + \nabla q - \nu \Delta v + \mu I_H  (v-u) - \gamma \nabla (\nabla \cdot v)  & = f, \label{contDA1}
\\ \nabla \cdot v & = 0, \label{contDA2} 
\end{align}
where $v$ is the approximate velocity and $q$ the pressure of this approximate flow. The viscosity $\nu >0$ and forcing $f$ are the same as the above. The scalar $\mu $ is known as the nudging parameter, and $I_H$ is the interpolation operator, where $H$ is the resolution of the coarse spacial mesh. The added data assimilation term forces (or nudges) the coarse spacial scales of the approximating solution $v$ to the coarse spacial scales of the true solution $u$. The initial value of $v$ is arbitrary. 

We note that in all computational studies discussed above, the equations have been handled with fully explicit schemes (typically forward Euler).  However, in explicit schemes, numerical instability is expected to arise from the term $\mu I_H  (v-u)$ on the right-hand side of \eqref{contDA2} for large values of $\mu$, and thus an implicit treatment of this term has advantages.  Thus, we study a backward Euler scheme for the data assimilation algorithm below.  Fully implicit schemes can be costly though, due to the need to solve nonlinear systems, which can require, e.g., expensive Newton solves at every time step (Newton methods have other theoretical problems, discussed below).  Therefore, we also study implicit-explicit (IMEX) schemes, which handle the nonlinear term semi-implicitly, but the linear terms (in particular, $\mu I_H  (v-u)$) implicitly.

In \cite{Olson_Titi_2008_TCFD}, it is argued (in the context of determining modes) that no higher-order Runge-Kutta-type methods or (fully) implicit methods of order greater than one can be constructed which satisfy the criteria of having the same discrete dynamics for $u$ and $v$, and which use only the information of $I_H(u)$ (as opposed to $u$) in the computation of $v$.  This is the reason why we use backward-differentiation methods, although Adams-Bashforth/Adams-Moulton would also be suitable choices.  We remark that, in the case of implicit methods, such methods do not make sense to use directly as one would need ``knowledge of the future;'' namely, $I_H(u^{n+1})$.  However, by interpreting our simulations as being run ``one time-step in the past,'' so that $I_H(u^{n+1})$ is taken to be the most recent data, not future data that is unmeasured.  The algorithms we propose in this work are consistent with the requirement stated in \cite{Olson_Titi_2008_TCFD} that the right-hand side of the assimilated system not be evaluated more than once per time step.  This is because the algorithms proposed here are only semi-implicit, and therefore do not require repeated solves due to the use of, e.g., Newton methods.
We also note that typically multi-step methods require initializing the first few steps via another method, such as a higher-order Runge-Kutta method.  However, we prove that for \textit{any} initialization of the first few steps, the solutions generated by the algorithm converge to the true solution.  For example, the first few steps could all be initialized to zero.  Thus, algorithms we present below have the advantage of needing no special scheme for the common problem of initializing a multi-step method.


This paper is organized as follows. In section 2, we will introduce the necessary notation and preliminary results needed in the proceeding sections. Section 3 introduces a linear first order scheme of the NSE with a grad-div term. We then show stability of the algorithms and optimal convergence rates of the data assimilation algorithm to the true NSE solution. Similarly, section 4 includes the convergence analysis of a linear second order numerical scheme of the NSE with a data assimilation term, under typical regularity assumptions of the NSE solution. Lastly, section 5 contains three numerical tests that illustrate the optimal convergence rates, and issues that arise in numerical implementation that one may not see from analysis of the scheme. 

\section{Notation and Preliminaries}
We consider a bounded open domain $\Omega \subset \R^d$ with $d$=2 or 3. The $L^2(\Omega)$ norm and inner product will be denoted by $\| \cdot \|$ and $(\cdot, \cdot)$, respectively, while all other norms will be labeled with subscripts. 

Denote the natural function spaces for velocity and pressure, respectively, by
\begin{align*}
X & := H^1_0(\Omega)^d \\
Q & := L^2_0(\Omega).
\end{align*}
In $X$, we have the Poincar\'e inequality: there exists a constant $C_P$ depending only on $\Omega$ such that for any $\phi\in X$,
\[
\| \phi \| \le C_P \| \nabla \phi \|.
\]
The dual norm of $X$ will be denoted by $\| \cdot \|_{-1}$.

We denote the trilinear form $b:X\times X\times X\rightarrow \mathbb{R}$, which is defined on smooth functions $u,v,w$ by
\[
b(u,v,w)=\frac12 (u\cdot\nabla v,w) - \frac12 (u\cdot\nabla w,v).
\]
An equivalent form of $b$ on $X\times X\times X$ can be constructed on smooth functions via
\[ 
b(u,v,w) = (u \cdot \nabla v,w) + \frac{1}{2}((\nabla \cdot u)v, w).  
\]
An important property of the $b$ operator is that $b(u,v,v)=0$ for $u,v\in X$.

We will utilize the following bounds on $b$.
\begin{lemma}\label{bbounds}
There exists a constant $M>0$ dependent only on $\Omega$ satisfying
\begin{align*}
|b(u,v,w)| & \le M \| \nabla u \| \| \nabla v \| \| \nabla w\|, \\
|b(u,v,w)| & \le M \| u \| (\| \nabla v \|_{L^3} + \| v \|_{L^{\infty}} ) \| \nabla w\|,
\end{align*}
for all $u,v,w\in X$ for which the norms on the right hand sides are finite.
\end{lemma}
{\color{black} \begin{remark}
Here and throughout, sharper estimates are possible if we restrict to 2D.  However, for simplicity and generality, we do not make this restriction.
	\end{remark}}
\begin{proof}
These well known bounds follow from H\"older's inequality, Sobolev inequalities, and the Poincar\'e inequality.
\end{proof}

\subsection{Discretization preliminaries}

Denote by $\tau_h$ a regular, conforming triangulation of the domain $\Omega$, and let $X_h \subset X$, $Q_h \subset Q$ be an inf-sup stable pair of discrete velocity - pressure spaces. For simplicity, we will take $X_h = X \cap P_k$ and $Q_h = Q \cap P_{k-1}$ Taylor-Hood or Scott-Vogelius elements
however our results in the following sections are extendable to most other inf-sup stable element choices.

We assume the mesh is sufficiently regular for the inverse inequality to hold:  there exists a constant $C$ such that for all $v_h\in X_h$,
\[
\| \nabla v_h \| \le C h^{-1} \| v_h \|.
\]
Define the discretely divergence free subspace by 
\[ 
V_h := \{ v_h \in X_h \,\, |\,\, (\nabla \cdot v_h, q_h) = 0 \,\, \forall \,\, q_h \in Q_h   \} . 
\]

We denote $I_H$ be an interpolation operator satisfying
\begin{align}
\| I_H  (\phi) - \phi \| &\le C_I h \| \nabla \phi\| \label{interp1}
\\ \|I_H(\phi)\| &\leq C \|\phi\| \label{interp2}
\end{align}
for some $C\geq1$, and for all $\phi\in X$.  Here, $H$ is a characteristic point spacing for the interpolant, and will satisfy $h\le H$, $H=ch$.  The spacing $H$ corresponds in practice to points where (true solution) measurements are taken, so $H$ should be as large as possible but still satisfying \eqref{interp1}-\eqref{interp2}.

Throughout this paper, we make the assumption on the mesh width $h$ that it satisfies the data dependent restriction
\[
h <  \sqrt{\frac{2\nu}{ C_I^2 C(data,u)}}.
\]
This will allow for choosing nudging parameters $\mu$ in the interval $(C(data,u),  \frac{\nu}{2}C_I^{-2}h^{-2})$.

We also define the quantity 
\[
\alpha := \nu - 2\mu C_I^2 h^2,
\]
and will assume that $\alpha>0$.  Note that $\mu$ will also have a data dependent lower bound, but choosing $h$ small enough will allow an appropriate $\mu$ to be chose.

\subsection{Additional preliminaries}

Several results in this paper utilize the following inequality for sequences.

\begin{lemma}\label{geoseries}
Suppose constants $r$ and $B$ satisfy $r>1$, $B\ge 0$.  Then if the sequence of real numbers $\{a_n\}$ satisfies 
\begin{align}
ra_{n+1} \le a_n + B, \label{sequence}
\end{align}
we have that
\[
a_{n+1} \le a_0\left(\frac{1}{r}\right)^{n+1}  + \frac{B }{r-1}.
\]
\end{lemma}
\begin{proof}
The inequality \eqref{sequence} can be written as 
\[
 a_{n+1}\leq \frac{a_n}{r} + \frac{B}{r}. 
 \]
 Recursively, we obtain
 \begin{align*}
 a_{n+1}&\leq  
 \frac{1}{r}\left(\frac{a_{n-1}}{r} + \frac{B}{r}\right)+ \frac{B}{r}
 \\ &=
 \frac{a_{n-1}}{r^2} + \frac{B}{r}\left(1 + \frac{1}{r}\right)
   \\ & \vdots
 \\ & \leq 
 \frac{a_0}{r^{n+1}} + \frac{B}{r}\left(1 + \frac{1}{r}+\cdots+\frac{1}{r^n}\right). 
\end{align*}
Now the resulting finite geometric series is bounded as 
\[
 \frac{B}{r}\left(1 + \frac{1}{r}+\cdots+\frac{1}{r^n}\right)  = \frac{B}{r}\cdot\frac{1-(1/r)^{n+1}}{1-(1/r)}
   \leq 
   \frac{B}{r}\cdot\frac{1}{1-(1/r)}
  \leq 
  \frac{B}{r-1}, 
\]
which gives the result. 
\end{proof}

The analysis in section 4 that uses a BDF2 approximation to the time derivative term will use the $G$-norm, which is commonly used in BDF2 analysis, see e.g. \cite{HW02}, \cite{CGSW13}.  Define the matrix
\begin{align*}
G = \begin{bmatrix}
1/2 & -1 \\ -1 & 5/2
\end{bmatrix},  
\end{align*}
and note that $G$ induces the norm $\|x\|_G^2 := (x, Gx)$, which is equivalent to the $(L^2)^2$ norm:
\[ C_l \|x\|_G \leq \|x\| \leq C_u \|x\|_G  \]
where $C_l=3-2\sqrt{2}$ and $C_l=3+2\sqrt{2}$.
The following property is well-known \cite{HW02}. Set $\chi_v^n := [v^{n-1}, v^n]^T$, if $v^i \in L^2(\Omega)$, $i = n-1, n$, we have 
\begin{align}
\left(\frac{1}{2}(3v^{n+1} - 4v^n + v^{n-1}), v^{n+1}\right) = \frac{1}{2}(\|\chi_v^{n+1}\|_G^{2} - \|\chi_v^n\|_G^2) + \frac{1}{4} \|v^{n+1} - 2v^n + v^{n-1}\|^2 .
\label{Gidentity}
\end{align}

\section{A first order IMEX-FEM scheme and its analysis}

We consider now an efficient fully discretized scheme for \eqref{contDA1}-\eqref{contDA2}.  We use
a first order temporal discretization for the purposes of simplicity of analysis, and in the next section we consider the extension
to second order time stepping.  The time stepping method employed is backward Euler, but linearized at each time step by lagging part of the 
convective term in time.  The spacial discretization is the finite element method, and we assume the velocity-pressure finite element
spaces $(X_h,Q_h)=(P_k,P_{k-1})$ for simplicity (although extension to any LBB-stable pair can be done without significant difficulty).  We 
also utilize grad-div stabilization, with parameter $\gamma>0$, and will assume $\gamma= O(1)$.  For most common element choices, grad-div stabilization is known to improve mass conservation and reduce the effect of the pressure on the velocity error \cite{JLMNR17}; a similar effect is observed in the convergence result for this DA scheme, as well as in the numerical tests.  In this section, we prove well-posedness of the scheme, as well as an error estimate, both of which are uniform in $n$ (global in time), 
provided some restrictions on the nudging parameter and on the time step size.

We now define the first order IMEX discrete DA algorithm for NSE.
\begin{algorithm} \label{alg1a}
	Given any initial condition $v_h^0\in V_h$, forcing $f\in L^{\infty}(0,\infty;L^2(\Omega)),$ and true solution $ u \in L^\infty (0,\infty;L^2(\Omega))$, find $(v_h^{n+1}, q_h^{n+1}) \in (X_h, Q_h)$ for $n = 0,1,2,...$, satisfying
	\begin{eqnarray}
	\frac{1}{\Delta t} \left( v_h^{n+1} - v_h^n,\chi_h \right) + b(v_h^{n},v_h^{n+1},\chi_h) - (q_h^{n+1},\nabla \cdot \chi_h) + \gamma (\nabla \cdot v_h^{n+1}, \nabla  \cdot \chi_h)&& \nonumber \\ + \nu (\nabla v_h^{n+1},\nabla \chi_h) + \mu (I_H(v_h^{n+1} - u^{n+1}),\chi_h)&= & (f^{n+1},\chi_h) \label{femda1_IMEX} \\
	(\nabla \cdot v_h^{n+1},r_h)  &=& 0, \label{femda2_IMEX}
	\end{eqnarray}
	for all $(\chi_h, r_h) \in X_h \times Q_h$. 
	\end{algorithm}
\begin{remark}
	Under some assumptions on the true solution $u$, this algorithm converges to $u$ as $t \to \infty$, independent of $v_h^0$, meaning the initial condition can be chosen arbitrarily. 
\end{remark}

We first prove that Algorithm \ref{alg1a} is well-posed, globally in time, without any restriction on the time step size $\Delta t$.
\begin{lemma} \label{L2stabilityBE}
	Suppose $\mu$, $h$ satisfy
	\[ 
	h< \frac{\sqrt{\nu}}{C_I\sqrt{2}}
	\text{ and }
	1\le \mu <  \frac{\nu}{2C_I^{2}h^{2}}.  
	\]
	Then  for any time step size $\Delta t>0$, Algorithm \ref{alg1a} is well-posed globally in time, and solutions are nonlinearly long-time stable: for any $n>0$,
	\begin{align*}
	\|v_h^{n}\|^2    
	 &\leq  
	 \left( \frac{1}{1+(\lambda+\mu)\Delta t  }\right)^{n} \|v_h^0\|^2 
	 + C( \nu^{-1} F + U) \le C(data),
	\end{align*}	
	where $\lambda = C_P^{-2}\alpha>0$, $F:=\| f \|_{L^{\infty}(0,\infty;H^{-1})}^2$, and
	$U:=\| u \|_{L^{\infty}(0,\infty;L^{2})}^2$.
\end{lemma} 


\begin{proof}
	Since the scheme is linear and finite dimensional, proving the stability bound in Lemma \ref{L2stabilityBE} will imply global-in-time well-posedness.

	We begin the proof for the stability bound by choosing $\chi_h = v_h^{n+1}$ in Algorithm \ref{alg1a}, which vanishes the pressure and nonlinear terms. We then add and subtract $v_h^{n+1}$ in the first component of the nudging term, which yields (after dropping the non-negative terms $\gamma \| \nabla \cdot v_h^{n+1}\|^2$ and $\frac{1}{2\Delta t} \|v_h^{n+1} - v_h^n\|^2$  on the left)
	\begin{align}
	\frac{1}{2\dt} [ \|v_h^{n+1}\|^2 - \|v_h^n\|^2 &  ] + \nu \|\nabla v_h^{n+1}\|^2 + \mu\|v_h^{n+1}\|^2 \nonumber
	\\ & \le (f^{n+1}, v_h^{n+1}) - \mu(I_H v_h^{n+1} - v_h^{n+1}, v_h^{n+1}) + \mu ( I_H   u^{n+1}, v_h^{n+1}).
	\label{stab1}
	\end{align}
	 The first term on the right hand side is bounded using the dual norm of $X$ and Young's inequality, which yields
	\begin{align*}
	(f^{n+1}, v_h^{n+1}) &\leq \|f^{n+1}\|_{-1}\|\nabla v_h^{n+1}\|
	 \leq \frac{ \nu^{-1}}{2} \|f^{n+1}\|_{-1}^2 + \frac{\nu}{2}\|\nabla v_h^{n+1}\|^2.
	\end{align*}
	The second term is bounded using Cauchy-Schwarz, interpolation property \eqref{interp1}, and Young's inequality, after which we have that 
	\begin{align*}
	\mu(I_H v_h^{n+1} - v_h^{n+1}, v_h^{n+1}) &\leq \mu\|I_H v_h^{n+1} - v_h^{n+1}\| \|v_h^{n+1}\|
	\\ & \leq \mu C_I^2 h^2 \|\nabla v_h^{n+1}\|^2 + \frac{\mu}{4}\|v_h^{n+1}\|^2.
	\end{align*}
	Finally, the last right hand side term will be bounded with these same inequalities, and property \eqref{interp2}, to obtain
	\begin{align*}
	\mu (I_Hu^{n+1}, v_h^{n+1}) &\leq \mu\|I_Hu^{n+1}\| \| v_h^{n+1}\|
	\\ & \leq C \mu \|u^{n+1}\|^2 + \frac{\mu}{4} \| v_h^{n+1}\|^2.
	\end{align*}
	Combine the above bounds for the right hand side of \eqref{stab1}, then multiply both sides by $2\Delta t$ and reduce, to get
	\begin{align}
	\|v_h^{n+1}\|^2 - \|v_h^n\|^2    +  \alpha \dt \|\nabla v_h^{n+1}\|^2 + \mu & \dt \|v_h^{n+1}\|^2 \nonumber 
	\\ & \leq  \nu^{-1} \dt\|f^{n+1}\|_{-1}^2 + C \mu\dt \|u^{n+1}\|^2,
	\label{stab2}
	\end{align}
	recalling that $\alpha= \nu - 2 \mu C_I^2h^2>0$.  Applying the Poincar\'e inequality to the viscous term, denoting $\lambda=C_P^{-2}\alpha$, and using the assumed regularity of $f$ and $u$ provides the bound
	\begin{align*}
	(1+(\lambda+\mu)\Delta t) \|v_h^{n+1}\|^2    
	 \leq  \|v_h^n\|^2 + \dt C( \nu^{-1} F +  \mu U). 
	\end{align*}
	Next apply Lemma \ref{geoseries} to find that 
	\begin{align*}
	\|v_h^{n+1}\|^2    
	 &\leq  
	 \left( \frac{1}{1+(\lambda+\mu)\Delta t  }\right)^{n+1} \|v_h^0\|^2 
	 + \frac{  \Delta t C( \nu^{-1} F +  \mu U)}{ \Delta t (\lambda + \mu) } \\
	 &\leq  
	 \left( \frac{1}{1+(\lambda+\mu)\Delta t  }\right)^{n+1} \|v_h^0\|^2 
	 + C(\lambda + \mu)^{-1} \nu^{-1} F + CU.
	\end{align*}
	Finally, since we assume $\mu\ge 1$, we obtain a bound for $v_h^{n+1}$, uniform in $n$:
	\begin{align}\label{Euler_stab_bound}
	\|v_h^{n+1}\|^2    
	 &\leq  
	 \left( \frac{1}{1+(\lambda+\mu)\Delta t  }\right)^{n+1} \|v_h^0\|^2 
	 + C( \nu^{-1} F + U) \le C(data).
	\end{align}	
	
	A similar stability bound can be used to show that solutions at each time step are unique, since the difference between two solutions satisfies the same bound as \eqref{Euler_stab_bound}, except with $F=U=0$.  Since the scheme is linear
	and finite dimensional at each time step, this also implies existence and uniqueness.  Finally, since the stability bound is
	uniform in $n$, we have global in time well-posedness of the scheme.
	\end{proof}

We will now prove that solutions to Algorithm \ref{alg1a} converge to the true NSE solution at a rate of $\dt + h^k$, globally in time, provided restrictions on $\dt$ and $\mu$ are satisfied.
\begin{theorem} \label{conv2}
	Let $u,p$ solve the NSE \eqref{nse1}-\eqref{nse2} with given $f\in L^{\infty}(0,\infty;L^2(\Omega))$ and $u_0\in L^2(\Omega)$, with $u \in L^\infty (0, \infty; H^{k+1}(\Omega))$, $p \in L^\infty (0, \infty; H^{k}(\Omega))$ ($k\ge 1$), $u_{t} \in L^{\infty}(0,\infty;L^2(\Omega))$, and $u_{tt} \in L^\infty(0, \infty; H^1(\Omega))$.  Denote $U:= | u |_{L^{\infty}(0,\infty;H^{k+1})}$ and $P:= | p |_{L^{\infty}(0,\infty;H^{k})}$.  Assume the time step size satisfies
\[
\Delta t <  CM^2\nu^{-1} \left( h^{2k-2} U^2  +   \|\nabla u^{n+1}\|^2_{L^3} + \|u^{n+1}\|_{L^\infty}^2 \right)^{-1},
\]
and the parameter $\mu$ satisfies
\[
 CM^2\nu^{-1} \bigg(  h^{2k-2} U^2 +  \|\nabla u^{n+1}\|_{L^3}^2 + \| u^{n+1}\|_{L^{\infty}}^2  \bigg) < \mu < \frac{2\nu}{C_I^2 h^2}.
\]
Then the error in solutions to Algorithm \ref{alg1a} satisfies, for any $n$,	
\[
\| u^n - v_h^n \|^2 \le \left( \frac{1}{1+2\lambda\Delta t} \right)^{n} \| u_0 - v_h^0 \|^2 + \frac{R}{\lambda},
\] 
where $\lambda = \alpha C_P^{-2}$ and
	\[
	R =C\bigg( (1+M^2)\nu^{-1} \dt^2 
	+  h^{2k} U^2 ( M^2\nu^{-1} + M^2\nu^{-1}h^{2k}U^2 + \nu + \gamma + \nu C_I^{-2} )\bigg).
	\]	
\end{theorem}
\begin{remark}
For the case of Taylor-Hood elements and initial condition $v_h^0\equiv0$ in the DA algorithm, the result of the theorem reduces to
\[
\| u^n - v_h^n \| \le C \left(
\left( \frac{1}{1+2\lambda\Delta t} \right)^{n/2} \| u_0 \|
+ 
\Delta t + h^k
\right).
\]
\end{remark}
\begin{remark}
The time step restriction is a consequence of the IMEX time stepping.  If we instead consider the fully nonlinear scheme, i.e. with $b(v_h^n,v_h^{n+1},\chi_h)$ replaced by $b(v_h^{n+1},v_h^{n+1},\chi_h)$, then 
the time restriction does not arise here, but arises in the analysis of well-posedness.
\end{remark}
\begin{remark}
Similar to the case of NSE-FEM without DA, grad-div stabilization reduces the effect of the pressure on the $L^2(\Omega)$ DA solution error.  With grad-div, the contribution of the error term is $h^{k} \gamma^{-1/2} 	 | p |_{L^{\infty}(0,\infty;H^{k})} $, but without it, the $\gamma^{-1/2}$ would be replaced by a $\nu^{-1/2}$.  If divergence-free elements were used, then this term would completely vanish.
\end{remark}

\begin{proof}
Throughout this proof, the constant $C$ will denote a generic constant, possibly changing at each instance, that is independent of $h$, $\mu$, and $\dt$.

	Using Taylor's theorem, the NSE (true) solution satisfies, for all  $\chi_h\in X_h$,
		\begin{multline}
	\frac{1}{\Delta t} \left( u^{n+1} - u^n,\chi_h \right) + b(u^{n},u^{n+1},\chi_h) - (p^{n+1},\nabla \cdot \chi_h) + \gamma (\nabla \cdot u^{n+1}, \nabla \cdot \chi_h)  \\ + \nu (\nabla u^{n+1},\nabla \chi_h)=  (f^{n+1},\chi_h)
	+ \frac{\dt}{2} (u_{tt}(t^*),\chi_h) + \dt b(u_t(t^{**}) ,u^{n+1},\chi_h), \label{nsetrue3}
		\end{multline}
	where $t^*, t^{**} \in [t^n, t^{n+1}]$.  Subtracting \eqref{femda1_IMEX} from \eqref{nsetrue3} yields the following difference equation, with $e^n:=u^n - v_h^n$:
		\begin{multline}
	\frac{1}{\Delta t} \left( e^{n+1} - e^n,\chi_h \right) + b(e^{n},u^{n+1},\chi_h) + b(v_h^n,e^{n+1},\chi_h) 
	- (p^{n+1} - q_h^{n+1},\nabla \cdot \chi_h) + \gamma (\nabla \cdot e^{n+1}, \nabla \cdot \chi_h)  \\ + \nu (\nabla e^{n+1},\nabla \chi_h) + \mu (I_H(e^{n+1}),\chi_h)=  
	 \frac{\dt}{2} (u_{tt}(t^*),\chi_h) + \dt b(u_t(t^{**}) ,u^{n+1},\chi_h). \label{daerr1}
		\end{multline}
	Denote $P_{V_h}^{L^2}(u^n)$ as the $L^2$ projection of $u^n$ into the discretely divergence-free space $V_h$.  We decompose the error as $e^n = (u^n - P_{V_h}^{L^2}(u^n)) + (P_{V_h}^{L^2}(u^n) - v_h^n) =: \eta^n + \phi_h^n$, and then choose
$\chi_h = \phi_h^{n+1}$, which yields
	\begin{align}
	\frac{1}{2\dt}[ \|\phi_h^{n+1}\|^2&  - \|\phi_h^n\|^2 + \|\phi_h^{n+1} - \phi_h^n\|^2 ] + \nu\|\nabla \phi_h^{n+1}\|^2 + \mu (I_H(\phi_h^{n+1}),\phi_h^{n+1})  + \gamma \|\nabla \cdot \phi_h^{n+1}\|^2 \nonumber
	\\ & = \frac{\dt}{2}(u_{tt}(t^*) , \phi_h^{n+1}) 
	+ \dt b(u_t(t^{**}), u^{n+1}, \phi_h^{n+1})
	- b( v_h^n, \eta^{n+1}, \phi_h^{n+1}) \nonumber
	\\ & \,\,\,\, 
	- b(\eta^n, u^{n+1}, \phi_h^{n+1})
	- b(\phi_h^n, u^{n+1}, \phi_h^{n+1})
	-  \mu (I_H(\eta^{n+1}),\phi_h^{n+1})
		\nonumber
	\\ & \,\,\,\, 
	+ (p^{n+1}-r_h, \nabla \cdot \phi_h^{n+1})
	- \nu (\nabla \eta^{n+1}, \nabla \phi_h^{n+1})
	- \gamma (\nabla \cdot \eta^{n+1}, \nabla \cdot \phi_h^{n+1}), \label{phidiff0}
	\end{align}
	where $r_h$ is chosen arbitrarily in $Q_h$ since $\phi_h^{n+1}$ is discretely divergence free. We also used that $b(v_h^{n},\phi_h^{n+1},\phi_h^{n+1})=0$ and $(\eta^{n+1} - \eta^n,\phi_h^{n+1})=0$.  

	Many of these terms are bounded in a similar manner as in the case of backward Euler FEM for NSE, for example as in \cite{GR86, laytonbook,T77}.  Using these techniques (which mainly consist of carefully constructed Young and Cauchy-Schwarz inequalities and Lemma \eqref{bbounds}), we majorize all right side terms except the fifth and sixth terms to obtain
	\begin{align*}
	\frac{1}{2\dt}[ \|\phi_h^{n+1}\|^2&  - \|\phi_h^n\|^2 + \|\phi_h^{n+1} - \phi_h^n\|^2 ] + \frac{11\nu}{16}  \|\nabla \phi_h^{n+1}\|^2 + \mu (I_H(\phi_h^{n+1}),\phi_h^{n+1})  + \frac{\gamma}{2} \|\nabla \cdot \phi_h^{n+1}\|^2 \nonumber
	\\ & \le
	 \nu^{-1} \dt^2  \| u_{tt}(t^*) \|_{-1}^2   
	 + 4M^2\nu^{-1} \dt^2  \|\nabla u_{t}(t^*) \|^2  \|\nabla u^n \|^2 
	 + 4M^2\nu^{-1} \| \nabla v_h^n \|^2 \| \nabla \eta^{n+1} \|^2 \nonumber
	\\ & \,\,\,\, 
	+ 4M^2 \nu^{-1} \| \nabla \eta^n \|^2 \| \nabla u^{n+1} \|^2
	- b(\phi_h^n, u^{n+1}, \phi_h^{n+1})
	-  \mu (I_H(\eta^{n+1}),\phi_h^{n+1})
		\nonumber
	\\ & \,\,\,\, 
	+ \gamma^{-1} \| p^{n+1}-r_h \|^2
	+ 4 \nu \| \nabla \eta^{n+1} \|^2
	+ \gamma^{-1} \|\nabla \cdot \eta^{n+1} \|^2.
	\end{align*}
Adding and subtracting $\phi_h^{n+1}$ to $I_H(\phi_h^{n+1})$, and using regularity assumptions on $u$, the bound reduces to  
	\begin{align}
	\frac{1}{2\dt}[ \|\phi_h^{n+1}\|^2&  - \|\phi_h^n\|^2 + \|\phi_h^{n+1} - \phi_h^n\|^2 ] + \frac{11\nu}{16} \|\nabla \phi_h^{n+1}\|^2 + \mu \| \phi_h^{n+1} \|^2   + \frac{\gamma}{2} \|\nabla \cdot \phi_h^{n+1}\|^2 \nonumber
	\\ & \le
	 C(1+M^2)\nu^{-1} \dt^2 
	 + 4M^2\nu^{-1} \| \nabla v_h^n \|^2 \| \nabla \eta^{n+1} \|^2 
	+ 4CM^2 \nu^{-1} \| \nabla \eta^n \|^2 
	 \nonumber
	\\ & \,\,\,\, 
	+ \gamma^{-1} \| p^{n+1}-r_h \|^2
	+ 4 \nu \| \nabla \eta^{n+1} \|^2
	+ \gamma \|\nabla \cdot \eta^{n+1} \|^2  \nonumber \\
	&\ \ \ 	+ | b(\phi_h^n, u^{n+1}, \phi_h^{n+1}) |
	+  \mu | (I_H(\eta^{n+1}),\phi_h^{n+1}) |
	+ \mu | (I_H(\phi_h^{n+1}) - \phi_h^{n+1},\phi_h^{n+1}) |
	. \label{phidiff2}
	\end{align}

To bound the third to last term in \eqref{phidiff2}, we begin by adding and subtracting $\phi_h^{n+1}$ to its first argument, and get
\begin{equation}
| b(\phi_h^n, u^{n+1}, \phi_h^{n+1}) | \le | b(\phi_h^{n+1} - \phi_h^n, u^{n+1}, \phi_h^{n+1}) | + | b(\phi_h^{n+1}, u^{n+1}, \phi_h^{n+1}) |. \label{dp1}
\end{equation}
For both terms in \eqref{dp1}, we use Lemma \ref{bbounds} and Young's inequality to obtain the bounds
	\begin{align*}
	| b(\phi_h^n - \phi_h^{n+1}, u^{n+1}, \phi_h^{n+1}) | 
	 & \leq M \|\phi_h^n - \phi_h^{n+1}\| \left( \| \nabla u^{n+1}\|_{L^3} + \|u^{n+1}\|_{L^\infty}\right)  \|\nabla \phi_h^{n+1}\| 
	\\ & 
	\leq 4M^2 \nu^{-1}\|\phi_h^{n+1} - \phi_h^{n}\|^2 (\|\nabla u^{n+1}\|^2_{L^3} + \|u^{n+1}\|_{L^\infty}^2 ) + \frac{\nu}{16} \|\nabla \phi_h^{n+1}\|^2, 
	\end{align*}
	and	
	\begin{align*}
	|b(\phi_h^{n+1}, u^{n+1}, \phi_h^{n+1})| & \leq  C \|\phi_h^{n+1}\| \left( \|\nabla u^{n+1}\|_{L^3} + \| u^{n+1}\|_{L^{\infty}} \right) \|\nabla \phi_h^{n+1}\| \\
	& \leq 4M^2 \nu^{-1} \| \phi_h^{n+1} \|^2 \left( \|\nabla u^{n+1}\|_{L^3}^2 + \| u^{n+1}\|_{L^{\infty}}^2 \right) + \frac{\nu}{16} \| \nabla \phi_h^{n+1}\|^2.
	\end{align*}
	For the second to last term in \eqref{phidiff2}, Cauchy-Schwarz and Young inequalities, along with \eqref{interp2}, imply that
	\begin{align*}
	\mu(I_H(\eta^{n+1}), \phi_h^{n+1}) &\leq \mu\|I_H(\eta^{n+1})\|\|\phi_h^{n+1}\| 
	\\ & \leq C \mu\| \eta^{n+1}\|^2 + \frac{\mu}{4} \|\phi_h^{n+1}\|^2.
	\end{align*}
	For the last term in \eqref{phidiff2}, we apply Cauchy-Schwarz and Young's inequalities and  \eqref{interp1}  to get 
	\begin{align*}
	\mu \big| (I_H(\phi_h^{n+1}) - \phi_h^{n+1}, \phi_h^{n+1})\big|  &\leq \mu\|I_H(\phi_h^{n+1})- \phi_h^{n+1}\|\|\phi_h^{n+1}\| 
	\\ & \leq \mu C_Ih\|\nabla \phi_h^{n+1}\|\|\phi_h^{n+1}\|
	\\ & \leq \mu C_I^2h^2 \|\nabla \phi_h^{n+1}\|^2 + \frac{\mu}{4}\|\phi_h^{n+1}\|^2.
	\end{align*}
Combining the above bounds, and recalling the definition of $\alpha$, reduces \eqref{phidiff2} to
	\begin{align}
	\frac{1}{2\dt}[ \|\phi_h^{n+1}\|^2&  - \|\phi_h^n\|^2 + \|\phi_h^{n+1} - \phi_h^n\|^2 ] +  \frac{\gamma}{2} \|\nabla \cdot \phi_h^{n+1}\|^2 \nonumber \\
	+  \frac{\alpha}{2} \|\nabla &\phi_h^{n+1}\|^2 
	+ \left( \frac{\mu}{2} - 4M^2\nu^{-1} \left( \|\nabla u^{n+1}\|_{L^3}^2 + \| u^{n+1}\|_{L^{\infty}}^2 \right) \right)  \| \phi_h^{n+1} \|^2   
	 \nonumber
	\\ & \le
	 C(1+M^2)\nu^{-1} \dt^2 
	 + 4M^2\nu^{-1} \| \nabla v_h^n \|^2 \| \nabla \eta^{n+1} \|^2 
	+  4CM^2\nu^{-1} \| \nabla \eta^n \|^2 
	 \nonumber
	\\ & \,\,\,\, 
	+ \gamma^{-1} \| p^{n+1}-r_h \|^2
	+ 4\nu \| \nabla \eta^{n+1} \|^2
	+ \gamma \|\nabla \cdot \eta^{n+1} \|^2  \nonumber \\
	&\ \ \ 	
	+ C \mu\| \eta^{n+1}\|^2 + 4M^2 \nu^{-1}\|\phi_h^{n+1} - \phi_h^{n}\|^2 (\|\nabla u^{n+1}\|^2_{L^3} + \|u^{n+1}\|_{L^\infty}^2 )
	. \label{phidiff2b}
	\end{align}

	It remains to estimate the term $4M^2\nu^{-1} \| \nabla v_h^n \|^2 \| \nabla \eta^{n+1} \|^2$.  By adding and subtracting $u^n$ to $v_h^n$ and using the triangle inequality, we note that
	\[
	\| \nabla v_h^n \| \le \| \nabla u^n \| + \| \nabla \eta^n \| + \| \nabla (\phi_h^{n+1} - \phi_h^n) \| + \| \nabla \phi_h^{n+1} \|.
\]
Using interpolation estimates, we obtain the bound
\begin{eqnarray*}
 \| \nabla v_h^n \|^2 \| \nabla \eta^{n+1} \|^2
& \le &
C\bigg( \| \nabla u^n \|^2 \| \nabla \eta^{n+1} \|^2
 +
  \| \nabla \eta^n \|^2 \| \nabla \eta^{n+1} \|^2 \\
&&  +
   \| \nabla (\phi_h^{n+1} - \phi_h^n) \|^2 \| \nabla \eta^{n+1} \|^2
   +
    \| \nabla \phi_h^{n+1} \|^2 \| \nabla \eta^{n+1} \|^2 \bigg) \\
    & \le &
    C\bigg( h^{2k} U^2 + h^{4k} U^4 
     + h^{2k-2} U^2   \| \phi_h^{n+1} - \phi_h^n  \|^2 
    + h^{2k-2} U^2 \| \phi_h^{n+1} \|^2  \bigg),
    \end{eqnarray*}
	where in the last step we used the inverse inequality.  Using this in \eqref{phidiff2b} gives us that
	\begin{align}
	\frac{1}{2\dt}[ \|\phi_h^{n+1}\|^2&  - \|\phi_h^n\|^2] +  \frac{\gamma}{2} \|\nabla \cdot \phi_h^{n+1}\|^2 +   \frac{\alpha}{2} \|\nabla \phi_h^{n+1}\|^2  \nonumber \\
	 + \|\phi_h^{n+1} - & \phi_h^n\|^2 \left(  \frac{1}{2\Delta t} -CM^2\nu^{-1} h^{2k-2} U^2  -     4M^2 \nu^{-1} \|\nabla u^{n+1}\|^2_{L^3} + \|u^{n+1}\|_{L^\infty}^2 \right)
	  \nonumber \\
	+ \bigg( \frac{\mu}{2} - & CM^2\nu^{-1} h^{2k-2} U^2 -4M^2\nu^{-1} \left( \|\nabla u^{n+1}\|_{L^3}^2 + \| u^{n+1}\|_{L^{\infty}}^2 \right) \bigg)  \| \phi_h^{n+1} \|^2   
	 \nonumber
	\\ & \le
	 C(1+M^2)\nu^{-1} \dt^2 
	 + CM^2\nu^{-1}h^{2k}U^2 
	 + CM^2\nu^{-1}h^{4k}U^4
	+  CM^2\nu^{-1} \| \nabla \eta^n \|^2 
	 \nonumber
	\\ & \,\,\,\, 
	+ \gamma^{-1} \| p^{n+1}-r_h \|^2
	+ 4\nu \| \nabla \eta^{n+1} \|^2
	+ \gamma \|\nabla \cdot \eta^{n+1} \|^2 
	+ C \frac{\nu}{C_I^2 h^2 } \| \eta^{n+1}\|^2 	. \label{phidiff2c}
	\end{align}
Using the assumptions on $\mu$ and $\dt$, we reduce the left hand side of \eqref{phidiff2c} by dropping non-negative terms and using the Poincar\'e inequality, and the right hand side using interpolation properties for $\eta^n$ to obtain
	\begin{align}
	\frac{1}{2\dt}&[ \|\phi_h^{n+1}\|^2  - \|\phi_h^n\|^2]   +  \frac{\alpha}{2} C_P^{-2} \| \phi_h^{n+1}\|^2  \nonumber \\
 & \le
	 C(1+M^2)\nu^{-1} \dt^2 
	 + CM^2\nu^{-1}h^{2k}U^2 
	 + CM^2\nu^{-1}h^{4k}U^4
	+  CM^2\nu^{-1} h^{2k}U^2 
	 \nonumber
	\\ & \,\,\,\, 
	+ C\gamma^{-1} h^{2k}P^2
	+ C\nu h^{2k} U^2
	+C \gamma h^{2k} U^2
	+ C \frac{\nu}{C_I^2 h^2 } h^{2k+2} U^2 \nonumber \\
& \le
	C\bigg( (1+M^2)\nu^{-1} \dt^2 
	+  h^{2k} U^2 ( M^2\nu^{-1} + M^2\nu^{-1}h^{2k}U^2 + \nu + \gamma + \nu C_I^{-2} )
	+ \gamma^{-1} h^{2k}P^2 \bigg).
	\label{phidiff2d}
	\end{align}
	Define the parameter $\lambda=\alpha C_P^{-2}$, and reduce \eqref{phidiff2d} to
			\begin{align}
	(1+\lambda \Delta t) \|\phi_h^{n+1}\|^2   
	 \le  
	 \|\phi_h^n\|^2 
	 + \Delta t R,
	 \label{phidiff2e}
	\end{align}
	where
	\[
	R =C\bigg( (1+M^2)\nu^{-1} \dt^2 
	+  h^{2k} U^2 ( M^2\nu^{-1} + M^2\nu^{-1}h^{2k}U^2 + \nu + \gamma + \nu C_I^{-2} )+ \gamma^{-1} h^{2k}P^2\bigg).
	\]	
Now using Lemma \ref{geoseries}, we obtain
\[
\| \phi_h^{n+1} \|^2 \le \left( \frac{1}{1+2\lambda\Delta t} \right)^{n+1} \| \phi_h^0 \|^2 + \frac{R}{\lambda}.
\] 
Applying the triangle inequality completes the proof.

\end{proof}

\section{Extension to a second order temporal discretization}
A BDF2 IMEX scheme for NSE with data assimilation is studied in this section.  We prove well-posedness, and global in time stability and convergence.  Similar results as the previous section are found, but here with second order temporal error.  The second order IMEX-FEM algorithm is defined as follows.

\begin{algorithm} \label{bdf2alg1}
	Given any initial conditions $v_h^0,\ v_h^{1} \in V_h$, forcing $f \in L^\infty(0,\infty; L^2(\Omega))$, and true solution $ u \in L^\infty(0,\infty; L^2(\Omega))$, find $(v_h^{n+1}, q_h^{n+1})$ $\in$ $(X_h, Q_h)$ for $n = 1,2,...$, satisfying
	\begin{eqnarray}
	\frac{1}{2\Delta t} \left( 3v_h^{n+1} - 4v_h^n + v_h^{n-1},\chi_h \right) + b(2v_h^{n} - v_h^{n-1},v_h^{n+1},\chi_h) - (q_h^{n+1},\nabla \cdot \chi_h) && \nonumber \\ + \gamma(\nabla \cdot v_h^{n+1}, \nabla \cdot \chi_h)+ \nu (\nabla v_h^{n+1},\nabla \chi_h) + \mu (I_H(v_h^{n+1} - u^{n+1}),\chi_h)&= & (f^{n+1},\chi_h), \label{femda1bdf2} \\
	(\nabla \cdot v_h^{n+1},r_h)  &=& 0 , \label{femda2bdf2}
	\end{eqnarray}
	for all $(\chi_h, r_h) \in X_h \times Q_h$, with $I_H$ a given interpolation operator satisfying \eqref{interp1}-\eqref{interp2}.
\end{algorithm}

We note that again the initial conditions can be chosen arbitrarily in $V_h$. Well-posedness of this algorithm, and long time stability, follow in a similar manner to the backward Euler case.  However, the treatment of the time
derivative terms is much more delicate, and we use G-stability theory to aid in this.  We state and prove the result now.

\begin{lemma} \label{L2stabilityBDF2}
Assume $h$ satisfies $0<h< \frac{\sqrt{\nu}}{C_I\sqrt{2}}$.  Then for any $\mu$ such that
	\[ 
	1\le \mu <  \frac{\nu}{2C_I^{2}h^{2}},  
	\]
	and  for any time step size $\Delta t>0$, Algorithm \ref{bdf2alg1} is well-posed globally in time, and solutions are nonlinearly long-time stable: for any $n>1$,
	\begin{align*}
 \bigg( C_l^2 &\left( \| v_h^{n+1} \|^2 + \| v_h^n \|^2 \right)  + \frac{\alpha\Delta t}{2} \| \nabla v_h^{n+1} \|^2 + \frac{\mu\Delta t}{4}\| v_h^{n+1} \|^2 \bigg) \\
&\le
 \left( C_u^2\| v_h^1 \|^2 + \| v_h^0 \|^2  + \frac{\alpha\Delta t}{2} \| \nabla v_h^{1} \|^2 + \frac{\mu\Delta t}{4}\| v_h^{1} \|^2 \right) \left( \frac{1}{1+\lambda\Delta t} \right)^{n+1}  + C\lambda^{-1} \nu^{-1}F^2 + C\lambda^{-1}\mu U^2.
\end{align*}
where $\lambda=\min \{2\Delta t^{-1}, \frac{\mu C_l^2}{4},\frac{\alpha C_P^{-2}C_l^2}{2} \}$,  $U:=\| u \|_{L^{\infty}(0,\infty;L^{2})}$, and $F:=\| f \|_{L^{\infty}(0,\infty;H^{-1})}$.
\end{lemma}

\begin{proof}
Choose $\chi_h=v_h^{n+1}$ in \eqref{femda1bdf2} and use \eqref{Gidentity} to obtain the bound
\begin{multline*}
\frac{1}{2\Delta t}\left( \| [v_h^{n+1};v_h^n] \|_G^2 \right)
+ \nu \| \nabla v_h^{n+1} \|^2
+ \mu(I_H(v_h^{n+1}),v_h^{n+1})
\\ \le
\frac{1}{2\Delta t}\left( \| [v_h^{n};v_h^{n-1}] \|_G^2 \right)
+ | \left(f^{n+1},v_h^{n+1} \right) |
+ \mu|(I_H(u^{n+1}),v_h^{n+1})|.
\end{multline*}
noting that we dropped the non-negative terms $\gamma \| \nabla \cdot v_h^{n+1} \|^2$ and $\frac{1}{4\Delta t} \| v_h^{n+1} - 2v_h^n + v_h^{n-1} \|^2$ from the left hand side, and that the nonlinear term and pressure term drop.  Analyzing the nudging, viscous and forcing terms just as in the backward Euler case, and multiplying both sides by $2\Delta t$ we reduce the bound to 
\begin{eqnarray*}
\| [v_h^{n+1};v_h^n] \|_G^2 
+ \alpha\Delta t \| \nabla v_h^{n+1} \|^2
+ \mu\Delta t \| v_h^{n+1} \|^2
\le
\| [v_h^{n};v_h^{n-1}] \|_G^2 
+ \Delta t( 2\nu^{-1} F^2 
+C\mu U^2).
\end{eqnarray*}
Next, drop the viscous term on the left hand side, and add $ \frac{\mu\Delta t}{4}\| v_h^{n} \|^2 + \frac{\alpha\Delta t}{2} \| \nabla v_h^{n} \|^2$ to both sides.  This gives
\begin{eqnarray*}
&& \hspace{-.5in} \left( \| [v_h^{n+1};v_h^n] \|_G^2  + \frac{\mu\Delta t}{4}\| v_h^{n+1} \|^2 +   \frac{\alpha\Delta t}{2} \| \nabla v_h^{n+1} \|^2  \right) \\
&& + \frac{\mu\Delta t}{4}  \left( \|v_h^{n+1} \|^2 + \| v_h^n \|^2 \right)  
+ \frac{\alpha\Delta t}{2} \left( \| \nabla v_h^{n+1} \|^2 + \| \nabla v_h^n \|^2 \right)
+ \frac{\mu\Delta t}{2} \| v_h^{n+1} \|^2
+ \alpha\Delta t \|  \nabla v_h^{n+1} \|^2 
\\
& & \le 
\left( \| [v_h^{n};v_h^{n-1}] \|_G^2  + \frac{\mu\Delta t}{4}\| v_h^{n} \|^2 +   \frac{\alpha\Delta t}{2} \| \nabla v_h^{n} \|^2 \right)
+ \Delta t( 2\nu^{-1} F^2+ C\mu U^2),
\end{eqnarray*}
which reduces using Poincar\'e's inequality and G-norm equivalence to
\begin{eqnarray*}
&& \hspace{-.5in} \left( \| [v_h^{n+1};v_h^n] \|_G^2  + \frac{\mu\Delta t}{4}\| v_h^{n+1} \|^2 +   \frac{\alpha\Delta t}{2} \| \nabla v_h^{n+1} \|^2  \right) \\
&& +  \frac{\mu\Delta tC_l^2 }{4}   \| [v_h^{n+1};v_h^n] \|_G^2 
+ \frac{\alpha\Delta t C_P^{-2} C_l^2}{2}  \| [v_h^{n+1};v_h^n] \|_G^2 
+ \frac{\mu\Delta t}{2} \| v_h^{n+1} \|^2
+ \alpha\Delta t \|  \nabla v_h^{n+1} \|^2 
\\
& &  \le 
\left( \| [v_h^{n};v_h^{n-1}] \|_G^2  + \frac{\mu\Delta t}{4}\| v_h^{n} \|^2 +   \frac{\alpha\Delta t}{2} \| \nabla v_h^{n} \|^2 \right)
+ \Delta t( 2\nu^{-1} F^2+ C\mu U^2),
\end{eqnarray*}
Thus there exists $\lambda=\min \{2\Delta t^{-1}, \frac{\mu C_l^2}{4},\frac{\alpha C_P^{-2}C_l^2}{2} \}$ such that
\begin{eqnarray*}
&& \hspace{-.5in} (1+\lambda\Delta t) \left( \| [v_h^{n+1};v_h^n] \|_G^2  + \frac{\alpha\Delta t}{2} \| \nabla v_h^{n+1} \|^2 + \frac{\mu\Delta t}{4}\| v_h^{n+1} \|^2 \right) \\
& & \le 
\left( \| [v_h^{n};v_h^{n-1}] \|_G^2 + \frac{\alpha\Delta t}{2} \| \nabla v_h^n \|^2 + \frac{\mu\Delta t}{4}\| v_h^{n} \|^2 \right)
+ \Delta t( 2\nu^{-1} F^2 + C\mu U^2),
\end{eqnarray*}
and so by Lemma \ref{geoseries},
\begin{align*}
 \bigg( \| [v_h^{n+1};v_h^n] \|_G^2  &+ \frac{\alpha\Delta t}{2} \| \nabla v_h^{n+1} \|^2 + \frac{\mu\Delta t}{4}\| v_h^{n+1} \|^2 \bigg) \\
&\le
 \left( \| [v_h^{1};v_h^0] \|_G^2  + \frac{\alpha\Delta t}{2} \| \nabla v_h^{1} \|^2 + \frac{\mu\Delta t}{4}\| v_h^{1} \|^2 \right) \left( \frac{1}{1+\lambda\Delta t} \right)^{n+1} \\
 & \  \ \ + C\lambda^{-1}(\nu^{-1}F^2 + \mu U^2).
\end{align*}
Applying the G-norm equivalence completes the proof of stability.

Since the scheme is linear and finite dimensional at each time step, this uniform in $n$ stability result gives existence and uniqueness of the algorithm at every time step.
\end{proof}

Proving a long time accuracy result for Algorithm \ref{bdf2alg1} follows in a similar manner to the first order result in the previous section.  The key difference is the time derivative terms, which we handle with the G-stability theory in a manner similar to the stability proof.

\begin{theorem} \label{bdf2conv2}
	Let $u,p$ solve the NSE \eqref{nse1}-\eqref{nse2} with given $f\in L^{\infty}(0,\infty;L^2(\Omega))$ and $u_0\in L^2(\Omega)$, with $u \in L^\infty (0, \infty; H^{k+1}(\Omega))$, $p \in L^\infty (0, \infty; H^{k}(\Omega))$ ($k\ge 1$), $u_{tt} \in L^{\infty}(0,\infty;L^2(\Omega))$, and $u_{ttt} \in L^\infty(0, \infty; H^1(\Omega))$.  Denote $U:= | u |_{L^{\infty}(0,\infty;H^{k+1})}$ and $P:= | p |_{L^{\infty}(0,\infty;H^{k})}$.  Assume the time step size satisfies
\[
\Delta t <  CM^2\nu^{-1} \left( h^{2k-3} U^2  +   \|\nabla u^{n+1}\|^2_{L^3} + \|u^{n+1}\|_{L^\infty}^2 \right)^{-1},
\]
and the parameter $\mu$ satisfies
\[
 CM^2\nu^{-1} \bigg(  h^{2k-3} U^2 +  \|\nabla u^{n+1}\|_{L^3}^2 + \| u^{n+1}\|_{L^{\infty}}^2  \bigg) < \mu < \frac{2\nu}{C_I^2 h^2}.
\]
Then there exists a $\lambda>0$ (independent of $h$ and $\Delta t$) such that the error in solutions to Algorithm \ref{bdf2alg1} satisfies, for any $n$,	
\[
\| u^n - v_h^n \|^2 \le \left( \frac{1}{1+\lambda\Delta t} \right)^{n} \| u_0 - v_h^0 \|^2 + \frac{R}{\lambda},
\] 
where 
	$
	R =C\nu^{-1}(1+M^2) \dt^4 
		+ C h^{2k} \bigg(  \gamma^{-1}  P^2
		+  (\nu + \gamma + M^2 \nu^{-1} + M^2 \nu^{-1}h^{2k}U^2 + \nu C_I^{-2} ) U^2 \bigg)$.	
\end{theorem}

\begin{remark}
For the case of Taylor-Hood $(P_2, P_1) $ or Scott-Vogelius $(P_2, P_1^{disc}) $ elements and $0$ initial condition in the DA algorithm, the result of the theorem reduces to
\[
\| u^n - v_h^n \| \le C \left(
\left( \frac{1}{1+\lambda\Delta t} \right)^{n/2} \| u_0 \|
+ 
\Delta t^2 + h^2
\right),
\]
{\color{black} where $C$ depends on problem data and the true solution, not $\dt$ or $h$. }
\end{remark}
\begin{remark}
Similar to the first order case, the time step restriction is a consequence of the IMEX time stepping.  If we instead consider the fully nonlinear scheme, then no $\Delta t$ restriction is required for a similar result to hold.  However, there is seemingly a time step restriction necessary for solution uniqueness for the nonlinear scheme.
\end{remark}
\begin{remark}
Just as in the first order case, grad-div stabilization reduces the effect of the pressure on the $L^2(\Omega)$ DA solution error.  With grad-div, the contribution of the error term is $h^{k} \gamma^{-1/2} 	 | p |_{L^{\infty}(0,\infty;H^{k})} $, but without it, the $\gamma^{-1/2}$ would be replaced by a $\nu^{-1/2}$.  If divergence-free elements were used, then this term would completely vanish.  We show in numerical experiment 2 below case where a DA simulation will fail with Taylor-Hood elements with $\gamma=0,1,10$, but will work very well with Scott-Vogelius elements.
\end{remark}

\begin{proof}
Throughout this proof, the constant $C$ will denote a generic constant, possibly changing from line to line, that is independent of $h$, $\mu$, and $\dt$.

	Using Taylor's theorem, the NSE (true) solution satisfies, for all  $\chi_h\in X_h$,
		\begin{multline}
	\frac{1}{2\Delta t} \left(3u^{n+1} - 4u^n + u^{n-1},\chi_h \right) + b(2u^{n}-u^{n-1},u^{n+1},\chi_h) - (p^{n+1},\nabla \cdot \chi_h) + \gamma (\nabla \cdot u^{n+1}, \nabla \cdot \chi_h)  \\ + \nu (\nabla u^{n+1},\nabla \chi_h)=  (f^{n+1},\chi_h)
	+ \frac{\dt^2}{3} (u_{ttt}(t^*),\chi_h) + \dt^2 b(u_{tt}(t^{**}) ,u^{n+1},\chi_h), \label{nsetrue3_new}
		\end{multline}
	where $t^*, t^{**} \in [t^{n-1}, t^{n+1}]$.  Subtracting \eqref{femda1bdf2} from \eqref{nsetrue3} yields the following difference equation, with $e^n:=u^n - v_h^n$:
		\begin{align*}
	\frac{1}{2\dt} & (3e^{n+1} - 4e^n + e^{n-1}, \chi_h ) + \nu (\nabla e^{n+1}, \nabla \chi_h) + \mu (I_H(e^{n+1}), \chi_h ) + \gamma (\nabla \cdot e^{n+1}, \nabla \cdot \chi_h)
	\\ & =  \frac{\dt^2}{3} (u_{ttt}(t^*), \chi_h) + \dt^2 (u_{tt}(t^{**}) \cdot \nabla u^{n+1}, \chi_h) - (p^{n+1} , \nabla \cdot \chi_h) + b(2v_h^n - v_h^{n-1}, e^{n+1}, \chi_h) 
	\\ & \,\,\,\, + b(2e^n - e^{n-1}, u^{n+1}, \chi_h) .
	\end{align*}
	We decompose the error into a piece inside the discrete space $V_h$ and one outside of it by adding and subtracting $P_{V_h}^{L^2}(u^n)$. Denote $\eta^{n} := u^{n} - P_{V_h}^{L^2}(u^n)$ and $\phi_h^{n} := P_{V_h}^{L^2}(u^n) - v_h^{n}$. Then $e^n = \eta^n + \phi_h^n$ with $\phi_h^n \in V_h$, and we choose $\chi_h=\phi_h^{n+1}$. Using identity \eqref{Gidentity} with $\psi_\phi := (\phi_h^n , \phi_h^{n+1})^T$, the difference equation becomes
	\begin{align}
		\frac{1}{2\dt} & [\|\psi_\phi^{n+1}\|_G^2 - \|\psi_\phi^{n}\|_G^2] + \frac{1}{4\dt} \|\phi_h^{n+1} - 2\phi_h^n + \phi_h^{n-1}\|^2 + \nu \|\nabla \phi_h^{n+1}\|^2 + \mu \|\phi_H^{n+1}\|^2 + \gamma \|\nabla \cdot \phi_h^{n+1}\|^2 \nonumber
		\\ & = \frac{\dt^2}{3} (u_{ttt}(t^*), \phi_h^{n+1}) 
		+ \dt^2 (u_{tt}(t^{**}) \cdot \nabla u^{n+1}, \phi_h^{n+1}) 
		- (p^{n+1} , \nabla \cdot \phi_h^{n+1}) 
		  \nonumber
		\\ & \,\,\,\,  
		+ b(2\phi_h^n - \phi_h^{n-1}, u^{n+1}, \phi_h^{n+1})  \nonumber
		+  b(2\eta^n - \eta^{n-1}, u^{n+1}, \phi_h^{n+1}) 
		+ b(2v_h^n - v_h^{n-1}, \eta^{n+1}, \phi_h^{n+1})
		\\ & \,\,\,\, - \nu (\nabla \eta^{n+1}, \nabla \phi_h^{n+1}) 
		- \mu (I_H\phi_h^{n+1} - \phi_h^{n+1}, \phi_h^{n+1}) 
		- \mu (I_H\eta^{n+1}, \phi_h^{n+1})   \nonumber
		\\ & \,\,\,\, - \gamma (\nabla \cdot \eta^{n+1}, \nabla \cdot \phi_h^{n+1}),\label{bdf2diff}
	\end{align}
	where we have added and subtracted $\phi_h^{n+1}$ in the interpolation term on the left hand side. We can now bound the right hand side of \eqref{bdf2diff}. 
	For the first nonlinear term in \eqref{bdf2diff}, we add and subtract $\phi_h^{n+1}$ in the first argument to get 
	\begin{align}
	b(2\phi_h^n - \phi_h^{n-1}, u^{n+1}, \phi_h^{n+1}) = b(\phi_h^{n+1}, u^{n+1}, \phi_h^{n+1}) - b(\phi_h^{n+1} - 2\phi_h^n + \phi_h^{n-1}, u^{n+1}, \phi_h^{n+1}). \label{nl3}
	\end{align}	
	We bound the two resulting terms using Lemma \ref{bbounds} and Young's inequality, via
	\begin{align*}
	b(\phi_h^{n+1}, u^{n+1}, \phi_h^{n+1}) &\leq CM \nu^{-1} (\|\nabla u^{n+1}\|_{L^3}^2 + \|u^{n+1}\|^2_{L^\infty}) \| \phi_h^{n+1}\|^2 + \frac{\nu}{16}\| \nabla\phi_h^{n+1}\|^2, 
	\end{align*}
	and
	\begin{align*}
	b(\phi_h^{n+1} - 2\phi_h^n + & \phi_h^{n-1}, u^{n+1}, \phi_h^{n+1})  \\
	&\leq CM \nu^{-1} (\|\nabla u^{n+1}\|_{L^3}^2 + \|u^{n+1}\|^2_{L^\infty}) \|\phi_h^{n+1} - 2\phi_h^n + \phi_h^{n-1}\|^2 + \frac{\nu}{16}\| \nabla\phi_h^{n+1}\|^2. 
	\end{align*}
	The second nonlinear term in \eqref{bdf2diff} is bounded with this same technique:
	\begin{align*}
	b(2\eta^n - \eta^{n-1},& u^{n+1}, \phi_h^{n+1})  \\
	&\leq CM^2 \nu^{-1} (\|\nabla u^{n+1}\|_{L^3}^2 + \|u^{n+1}\|^2_{L^\infty}) \|2\eta^n - \eta^{n-1}\|^2 + \frac{\nu}{16}\| \nabla\phi_h^{n+1}\|^2. 
	\end{align*}
	The last nonlinear term in \eqref{bdf2diff} requires a bit more work, and we start by adding and subtracting $2u^n - u^{n-1}$ in the first component, 
	which yields
	\begin{align}
	b(2v_h^n - v_h^{n-1}, \eta^{n+1}, \phi_h^{n+1}) &= b(2u^n - u^{n-1}, \nonumber \eta^{n+1}, \phi_h^{n+1}) + b(2e^n - e^{n-1}, \eta^{n+1}, \phi_h^{n+1})
	\\ &= b(2u^n - u^{n-1}, \eta^{n+1}, \phi_h^{n+1}) + b(2\phi_h^n - \phi_h^{n-1}, \eta^{n+1}, \phi_h^{n+1}) \nonumber
	\\ & \,\,\,\, +  b(2\eta^n - \eta^{n-1}, \eta^{n+1}, \phi_h^{n+1}). \label{nl4}
	\end{align}
	The first and third terms on the right hand side of \eqref{nl4} are bounded in the same way, using Lemma \ref{bbounds} and Young's inequality, to get that
	\begin{align*}
	b(2u^n - u^{n-1}, \eta^{n+1}, \phi_h^{n+1}) 
	& \le
	C\nu^{-1} M^2 \| \nabla ( 2u^n - u^{n-1} ) \|^2 \| \nabla \eta^{n+1} \|^2 + \frac{\nu}{16} \| \nabla \phi_h^{n+1} \|^2,
	\end{align*}
	\begin{align*}
	b(2\eta^n - \eta^{n-1}, \eta^{n+1}, \phi_h^{n+1}) 
	& \le
	C\nu^{-1} M^2 \| \nabla ( 2\eta^n - \eta^{n-1} ) \|^2 \| \nabla \eta^{n+1} \|^2 + \frac{\nu}{16} \| \nabla \phi_h^{n+1} \|^2,
	\end{align*}
	For the second term in \eqref{nl4} we first add $\phi_h^{n+1}$ to the first argument to get
	\begin{align*}
	b(2\phi_h^n - \phi_h^{n-1}, \eta^{n+1}, \phi_h^{n+1}) 
	=
	b(\phi_h^{n+1} - 2\phi_h^n + \phi_h^{n-1}, \eta^{n+1}, \phi_h^{n+1}) 
	+
	b(\phi_h^{n+1}, \eta^{n+1}, \phi_h^{n+1}),
	\end{align*}
	and then bound each resulting term using Lemma \ref{bbounds} and Young's inequality:
	\begin{align*}
	b(\phi_h^{n+1}, \eta^{n+1}, \phi_h^{n+1})
	&\le
	CM^2\nu^{-1} (\| \eta^{n+1} \|_{L^{\infty}}^2 + \| \nabla \eta^{n+1} \|_{L^3}^2 ) \| \phi_h^{n+1} \|^2 + \frac{\nu}{16} \| \nabla \phi_h^{n+1} \|^2,
	\end{align*}
	\begin{align*}
	b(\phi_h^{n+1} - 2\phi_h^n + \phi_h^{n-1}, & \eta^{n+1}, \phi_h^{n+1}) \\
	&\le
	CM^2\nu^{-1} (\| \eta^{n+1} \|_{L^{\infty}}^2 + \| \nabla \eta^{n+1} \|_{L^3}^2 ) \|\phi_h^{n+1} - 2\phi_h^n + \phi_h^{n-1} \|^2 + \frac{\nu}{16} \| \nabla \phi_h^{n+1} \|^2.
	\end{align*}
	
	Collecting the above bounds, we can reduce \eqref{bdf2diff} to
		\begin{align}
		\frac{1}{2\dt} & [\|\psi_\phi^{n+1}\|_G^2 - \|\psi_\phi^{n}\|_G^2] 
		+ \frac{9\nu}{16} \|\nabla \phi_h^{n+1}\|^2 
		+ \gamma \|\nabla \cdot \phi_h^{n+1}\|^2 \nonumber \\
		+ \bigg( \frac{1}{4\dt} & -CM^2\nu^{-1} (\| \eta^{n+1} \|_{L^{\infty}}^2 + \| \nabla \eta^{n+1} \|_{L^3}^2 + \| u^{n+1} \|_{L^{\infty}}^2 + \| \nabla u^{n+1} \|^2 ) \bigg) \|\phi_h^{n+1} - 2\phi_h^n + \phi_h^{n-1}\|^2  \nonumber \\
		+ \bigg( \mu & - CM^2\nu^{-1} (\| \eta^{n+1} \|_{L^{\infty}}^2 + \| \nabla \eta^{n+1} \|_{L^3}^2 + \| u^{n+1} \|_{L^{\infty}}^2 + \| \nabla u^{n+1} \|^2) \bigg) \|\phi_h^{n+1}\|^2 \nonumber \\
		& \le C\dt^2 \| u_{ttt} \|_{L^{\infty}(t^{n-1},t^{n+1},L^2)}  \|  \phi_h^{n+1}\| 
		+ \dt^2 | (u_{tt}(t^{**}) \cdot \nabla u^{n+1}, \phi_h^{n+1}) |
		+ | (p^{n+1}-r_h , \nabla \cdot \phi_h^{n+1}) |
		  \nonumber
		\\ & + \nu |(\nabla \eta^{n+1}, \nabla \phi_h^{n+1}) |
		+ \mu |(I_H\phi_h^{n+1} - \phi_h^{n+1}, \phi_h^{n+1})| 
		+ \mu |(I_H \eta^{n+1}, \phi_h^{n+1}) | 
		 \nonumber \\
		& +C\nu^{-1} M^2 \| \nabla ( 2u^n - u^{n-1} ) \|^2 \| \nabla \eta^{n+1} \|^2
		+C\nu^{-1} M^2 \| \nabla ( 2\eta^n - \eta^{n-1} ) \|^2 \| \nabla \eta^{n+1} \|^2 \nonumber \\
		& + CM^2 \nu^{-1} (\|\nabla u^{n+1}\|_{L^3}^2 + \|u^{n+1}\|^2_{L^\infty}) \|2\eta^n - \eta^{n-1}\|^2
		+ \gamma |(\nabla \cdot \eta^{n+1}, \nabla \cdot \phi_h^{n+1})|,
		\label{bdf2diff1}
	\end{align}
	where $r_h \in Q_h$ is chosen arbitrarily, see e.g. \cite{BS08}.
Now using interpolation estimates (and implicitly also the inverse inequality) along with regularity assumptions, we obtain
		\begin{align}
		\frac{1}{2\dt} & [\|\psi_\phi^{n+1}\|_G^2 - \|\psi_\phi^{n}\|_G^2] 
		+ \frac{9\nu}{16} \|\nabla \phi_h^{n+1}\|^2 
		+ \gamma \|\nabla \cdot \phi_h^{n+1}\|^2 \nonumber \\
		+ \bigg( \frac{1}{4\dt} & -CM^2\nu^{-1} ( h^{2k-3} U^2 + \| u^{n+1} \|_{L^{\infty}}^2 + \| \nabla u^{n+1} \|_{L^3}^2 ) \bigg) \|\phi_h^{n+1} - 2\phi_h^n + \phi_h^{n-1}\|^2  \nonumber \\
		+ \bigg( \mu & - CM^2\nu^{-1} (h^{2k-3}U^2 + \| u^{n+1} \|_{L^{\infty}}^2 + \| \nabla u^{n+1} \|_{L^3}^2) \bigg) \|\phi_h^{n+1}\|^2 \nonumber \\
		& \le C\dt^2 \| u_{ttt} \|_{L^{\infty}(t^{n-1},t^{n+1},L^2)}  \|  \phi_h^{n+1}\| 
		+ \dt^2| (u_{tt}(t^{**}) \cdot \nabla u^{n+1}, \phi_h^{n+1}) |
		+ |(p^{n+1} , \nabla \cdot \phi_h^{n+1}) |
		  \nonumber
		\\ & + \nu| (\nabla \eta^{n+1}, \nabla \phi_h^{n+1}) |
		 + \mu | (I_H \phi_h^{n+1} - \phi_h^{n+1}, \phi_h^{n+1}) |
		 + \mu | (I_H \eta^{n+1}, \phi_h^{n+1}) |  \nonumber
		 \nonumber \\
		& + CM^2 \nu^{-1} h^{2k}U^2 (1 + h^{2k}U^2)
		+ \gamma |(\nabla \cdot \eta^{n+1}, \nabla \cdot \phi_h^{n+1})| .
		\label{bdf2diff2}
	\end{align}
	Next we use the assumptions on $\Delta t$ and $\mu$, and apply bounds to the remaining right hand side terms just as in the backward Euler proof to get 
\begin{align}
		\frac{1}{2\dt} & [\|\psi_\phi^{n+1}\|_G^2 - \|\psi_\phi^{n}\|_G^2] 
		+ \alpha \|\nabla \phi_h^{n+1}\|^2 
		+ \frac{\gamma}{2} \|\nabla \cdot \phi_h^{n+1}\|^2 \nonumber \\
		& \le C\nu^{-1}(1+M^2) \dt^4 
		+ C h^{2k} \bigg(  \gamma^{-1}  P^2
		+  (\nu + \gamma + M^2 \nu^{-1} + M^2 \nu^{-1}h^{2k}U^2 + \nu C_I^{-2} ) U^2 \bigg) \nonumber \\
		& = R.
		\label{bdf2diff3}
\end{align}
This implies, with Poincar\'e's inequality that
\[
\|\psi_\phi^{n+1}\|_G^2 + 2C_l^2 \Delta t \alpha C_P^{-2} \| \phi_h^{n+1} \|^2 
\le 
\|\psi_\phi^{n}\|_G^2 
+ \Delta t R.
\]
From here, we can proceed just as in to the BDF2 long time stability proof above to obtain
\[
\|\psi_\phi^{n+1}\|_G^2 \le \|\psi_\phi^{0}\|_G^2 \left( \frac{1}{1+\lambda\Delta t} \right)^{n+1} +  \frac{R}{\lambda},
\]
and now the triangle inequality and G-norm equivalence complete the proof.
\end{proof}

\section{Numerical Experiments}

We now present results of three numerical tests that illustrate the theory above, and also show the importance of a careful 
choice of discretization.  That is, while the DA theory at the PDE level is critical, in a discretization there are additional consideration
and restrictions that can make the difference of a simulation succeeding or failing.   All of our tests use Algorithm \ref{bdf2alg1}, i.e. the BDF2 IMEX-FEM algorithm studied in section 4.

\subsection{Numerical Experiment 1: Convergence to an analytical solution}

For our first experiment, we illustrate the convergence theory above for Algorithm \ref{bdf2alg1} to the chosen analytical solution on $\Omega=(0,1)^2$,
\begin{align*}
u(x,y,t)& = (\cos(y+t), \sin(x-t))^T, 
\\ p(x,y,t) &= \sin(2\pi(x+t)).
\end{align*}
We take $\nu=0.01$, the forcing function $f$ is calculated from the continuous NSE, $\nu$, and the solution, and the initial velocity is taken to be $u_0=u(x,y,0)$.


We compute on a uniform mesh with Taylor-Hood elements and Dirichlet boundary conditions, and for simplicity we take $\gamma=0$, since the grad-div stabilization has little effect in this test problem.  The interpolation operator $I_H$ is chosen to be the nodal interpolant onto constant functions on the same mesh used for velocity and pressure, and the initial conditions for the DA algorithm are set to zero.

\begin{figure}[!ht]
\begin{center}
	\includegraphics[width = .67\textwidth, height=.45\textwidth,viewport=0 0 725 500, clip]{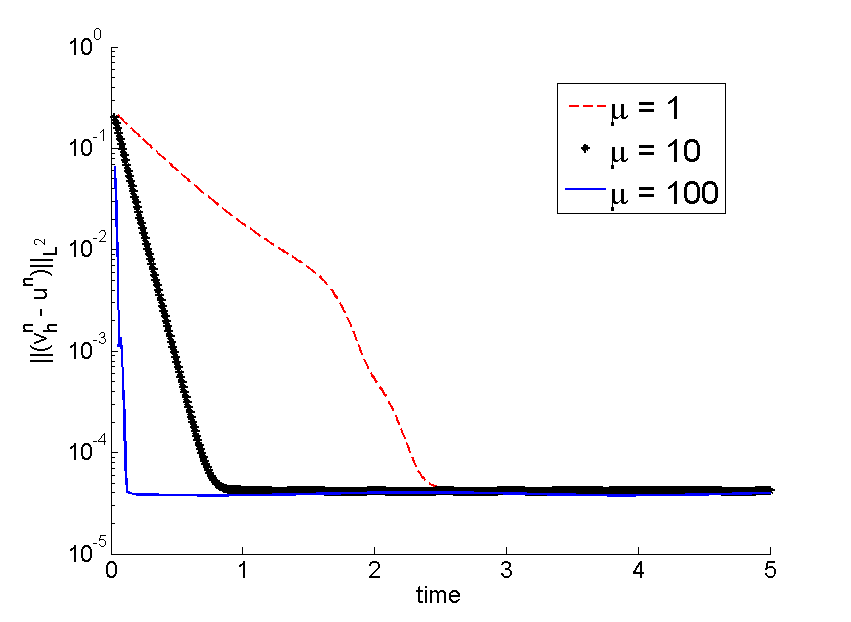}
	\caption{\label{MU}Shown above is a log-linear plot of convergence of the DA computed solutions to the true solution with increasing time $t$, for varying choices of the nudging parameter $\mu$. }
\end{center}
\end{figure}


Results are shown in Figure \ref{MU}, for $\mu=1,\ 10,\ 100$ using $h=\frac{1}{32}$ and $\Delta t=0.01$, by plotting the $L^2(\Omega)$ difference between the DA computed solution and the true solution versus time.  We observe convergence up to about $10^{-4}$, which is the level of the discretization error for the chosen time step and mesh size.  We observe that for larger choices of $\mu$, convergence to the true solution (up to discretization error) is much faster.  However, we note that the long time accuracy is not affected by $\mu$.

Table \ref{hconv} displays the convergence rates of Algorithm \ref{bdf2alg1} solutions to the true solution; error is calculated using the $L^2(\Omega)$ norm at the final time.  For these calculations, we take $\mu = 10$ and $\gamma = 1.0$ and run to an end time of $T = 4.0$. When observing the spacial convergence rates, we fix $\dt = 0.001$ and vary $h$, while for the temporal error we fix $h=\frac{1}{64}$ and vary $\Delta t$. We also test spacial and temporal convergence together, by reducing $h$ and $\Delta t$, but keeping the ratio $4h=\Delta t$.  In all cases we observe second order convergence for spacial and temporal error, which is consistent with our analysis.

\begin{table}[h!]
	{\small
		\begin{minipage}[b]{0.25\linewidth}
			\centering
			\begin{tabular}{|c | c | c |}
				\hline
				$h$ & Error & Rate \\ \hline
				1/4 & 4.12E-3 & - \\ \hline
				1/8 & 5.16E-4 & 3.00 \\ \hline
				1/16 & 5.91E-5 & 3.13 \\ \hline
				1/32 & 8.71E-6 & 2.76 \\ \hline
				1/64 & 1.92E-6 & 2.18 \\ \hline
				1/128 & 4.75E-7 & 2.02 \\ \hline
			\end{tabular}
		\end{minipage}
		\hspace{1cm}
		\begin{minipage}[b]{0.25\linewidth}
			\centering
			\begin{tabular}{| c | c | c |} 
				\hline
				$\dt$	&	Error 	& 	Rate	\\ \hline 	
				1	&	2.60E-3	& 	- 	\\ \hline	
				1/2	&	3.63E-4	& 	2.84	\\ \hline	
				1/4	&	6.84E-5	& 	2.41	\\ \hline	
				1/8	&	1.52E-5	& 	2.17	  \\ \hline	
				1/16	&	3.76E-6	& 	2.02	\\ \hline	
				1/32	&	1.09E-6	& 	1.78	 \\ \hline	
			\end{tabular}
		\end{minipage}
		\hspace{0.7cm}
		\begin{minipage}[b]{0.3\linewidth}
			\centering
			\begin{tabular}{| c | c | c | c |}
				\hline
				$h$	&	$\dt$	&	Error	&	Rate	\\ \hline	
				1/4   	&	1	&	4.69E-3	&	-	\\ \hline	
				1/8   	&	1/2	&	5.79E-4	&	3.02	\\ \hline	
				1/16   	&	1/4	&	9.16E-5	&	2.66	\\ \hline	
				1/32   	&	1/8	&	1.83E-5	&	2.32	\\ \hline	
				1/64   	&	1/16	&	4.38E-6	&	2.06	\\ \hline	
				1/128   	&	1/32	&	1.09E-6	&	2.00	\\ \hline	
			\end{tabular}
	\end{minipage}}
	\caption{\label{hconv} Convergence rates of Algorithm \ref{bdf2alg1} to the true solution with decreasing $h$ and fixed $\Delta t$ (left), fix $h$ and decreasing $\Delta t$ (middle), and also decreasing $h$ and $\Delta t$ at the same rate with $\Delta t=4h$ (right).}
\end{table}

\subsection{Numerical Experiment 2: The no-flow test and pressure-robustness}

For our second test, we show how pressure robustness of the discretization can have a dramatic
impact on the DA solution.  The test problem we consider is the so-called `no-flow test', where the forcing function of 
the NSE is given by $Ra(0,y)^T$, where $Ra>0$ is a dimensionless constant (the Rayleigh number), and with $Pr>0$ denoting the dimensionless Prandtl number:
\begin{align}
\frac{1}{Pr}( u_t + u\cdot \nabla u  ) + \nabla p - \Delta u &= Ra(0,y)^T, \label{NF1} \\
\nabla \cdot u &=0, \label{NF2} \\
u |_{\partial\Omega} &= 0. \label{NF3}
\end{align}
This test problem corresponds to the physical situations of temperature driven flow (i.e. the Boussinesq system), with the temperature $\theta$ profile specified to be stratified, i.e. $f=Ra \theta e_2$ with $\theta=y$.  Linear stratification is a natural steady state temperature profile.  Since the forcing is potential, the solution to the system \eqref{NF1}-\eqref{NF3} with $u_0=0$ initial condition is given by
\[
u=0, \ p = \frac{Ra}{2} y^2,
\]
for any $Pr>0$, hence the name no-flow.

We consider Algorithm \ref{bdf2alg1}, applied to the no-flow test with $Pr=1$ and $Ra=10^5$ (although this may seem like a large choice of a constant, for Boussinesq problems of practical interest, this choice of $Ra$ is actually quite small).  We use both Scott-Vogelius (SV) elements and Taylor-Hood (TH) elements, on a barycenter refined uniform discretization of the unit square with $h=\frac{1}{32}$.  With TH elements, we use $\gamma=0,1,10$.
We take $I_H$ to be the nodal interpolant in $X_h$, and nudging parameter $\mu=0.1$.  The time step size is chosen to be $\Delta t=0.025$, and solutions are computed up to end time $T=0.8$, using the $X_h$ interpolant of $(x\cos y,-\sin y)^T$ for $v_h^0$, and $v_h^1$ is calculated from taking one step of the backward Euler DA scheme.

\begin{figure}[!ht]
	\centering
	\includegraphics[width = .7\textwidth, height=.4\textwidth,viewport=0 0 600 450, clip]{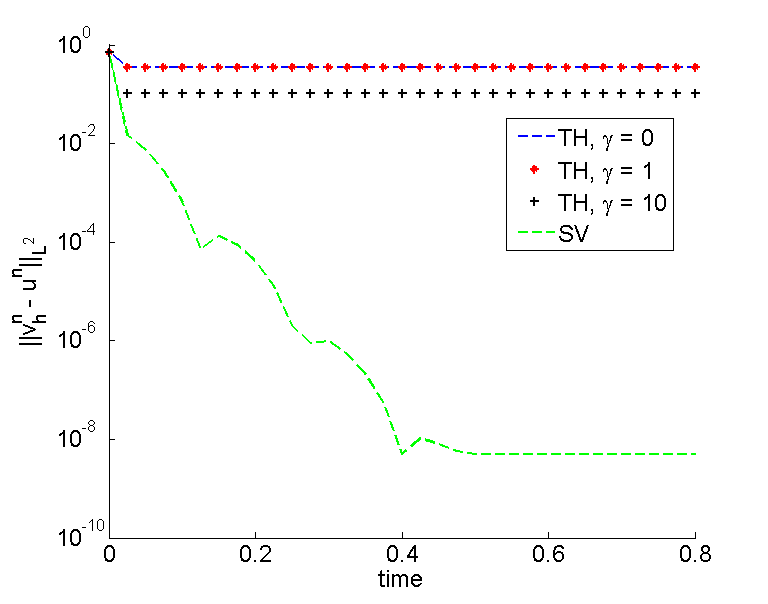}
	\caption{\label{Ra5mu1} Shown above is error in DA solutions for the no-flow solution, with SV element and TH elements (with varying $\gamma$).  TH elements yield an error on the order of $10^-1$, which we consider to be non-convergent, while SV elements perform seven orders of magnitude better than TH, converging with an error on the order of $10^{-8}$ after 20 time steps.}
\end{figure}


Results of the simulations are shown in Figure \ref{Ra5mu1}, as $L^2(\Omega)$ error versus time, and we observe a dramatic difference between solutions from the two element choices.  For TH elements, the results are poor due to the large pressure, which adversely affects the velocity error, even using the identity operator, which is the best possible choice of interpolation operator.
With $\gamma=10$, the TH solution has slightly lower error, however, it is still on the order of $10^{-1}$ accuracy, which we consider to be non-convergent.  On the other hand, the SV solution performs orders of magnitude better, decaying roughly at an exponential rate in time until it reaches a level around $10^{-8}$ and stays there.  Thus we observe here that in DA algorithms, 
element choice can be critical for obtaining accurate results in certain simulations, at least those of Boussinesq type.

\subsection{Numerical Experiment 3: 2D channel flow past a cylinder}

For our last experiment, we consider Algorithm \ref{bdf2alg1} applied to the common benchmark problem of 2D channel flow 
past a cylinder with Reynolds number 100  \cite{ST96}.  The domain is a $2.2\times0.41$ rectangular channel with a cylinder of radius $0.05$ centered at $(0.2,0.2)$, see Figure \ref{cylinderdomain}. There is no external forcing, the kinematic viscosity is taken to be $\nu=0.001$, no-slip boundary conditions are prescribed for the walls and the cylinder, while the inflow and outflow profiles are given by 
\begin{align*}
u_1(0,y,t) & = u_1(2.2,y,t) = \frac{6}{0.41^2}y(0.41-y),
\\u_2(0,y,t) & = u_2(2.2,y,t) = 0.
\end{align*}
\begin{figure}[!ht]
	\centering
	\includegraphics[scale = .8]{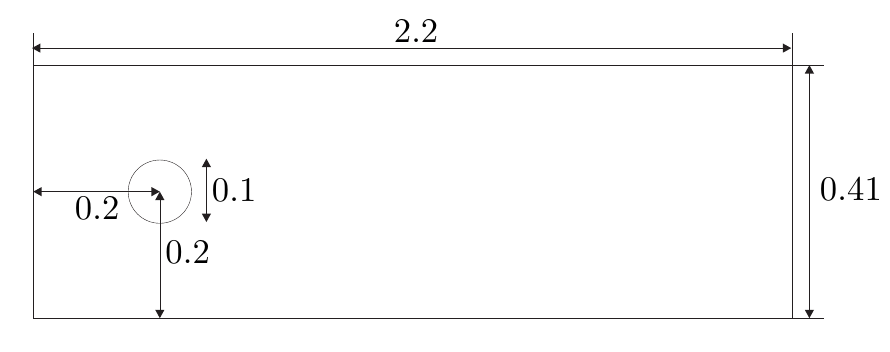}
	\caption{\label{cylinderdomain} Shown above is the domain for the flow past a cylinder test problem.}
\end{figure}

Since we do not have access to a true solution for this problem, we instead use a computed solution.  It is computed using the same BDF2-IMEX-FEM scheme as in Algorithm \ref{bdf2alg1} but without nudging (i.e. $\mu=0$), using $(P_2,P_1^{disc})$ Scott-Vogelius elements on a barycenter refined Delaunay mesh that provides 60,994 total degrees of freedom, a time step of $\Delta t=0.002$, and with the simulation starting from rest ($u_h^0=u_h^{-1}=0$).  We will refer to this solution as the DNS solution.  Lift and drag calculations were performed for the computed solution and compared to the literature \cite{ST96,MRXI17}, which verified the accuracy of the DNS.

For the lift and drag calculations, we used the formulas
\begin{align*}
c_d(t) &= 20\int_S \left( \nu \frac{\partial u_{t_S}(t)}{\partial n}n_y - p(t)n_x \right)dS,
\\c_l(t) &= 20 \int_S \left( \nu \frac{\partial u_{t_S}(t)}{\partial n}n_x - p(t)n_y \right)dS,
\end{align*}
where $p(t)$ is the pressure, $u_{t_S}$ the tangential velocity $S$ the cylinder, and $n = \langle n_x, n_y\rangle$ the outward unit normal to the domain. For calculations, we use the global integral formula from \cite{J04}.

For the DA algorithm, we start from $v_h^1 = v_h^0=0$, choose $\mu=10$, use the same spacial and temporal discretization parameters as the DNS, and begin assimilating with t=5 DNS solution (so time 0 for DA corresponds to t=5 for the DNS).   For the interpolant, we use constant interpolation on a mesh that is one refinement coarser, i.e. on the Delaunay mesh without the barycenter refinement.  The number of velocity degrees of freedom for constant functions on the coarse mesh is just 5,772.   The simulation is run on [0,5] (so the actual corresponding times for the DNS would be [5,10]).

\begin{figure}[!ht]
	\centering
	\includegraphics[width = .7\textwidth]{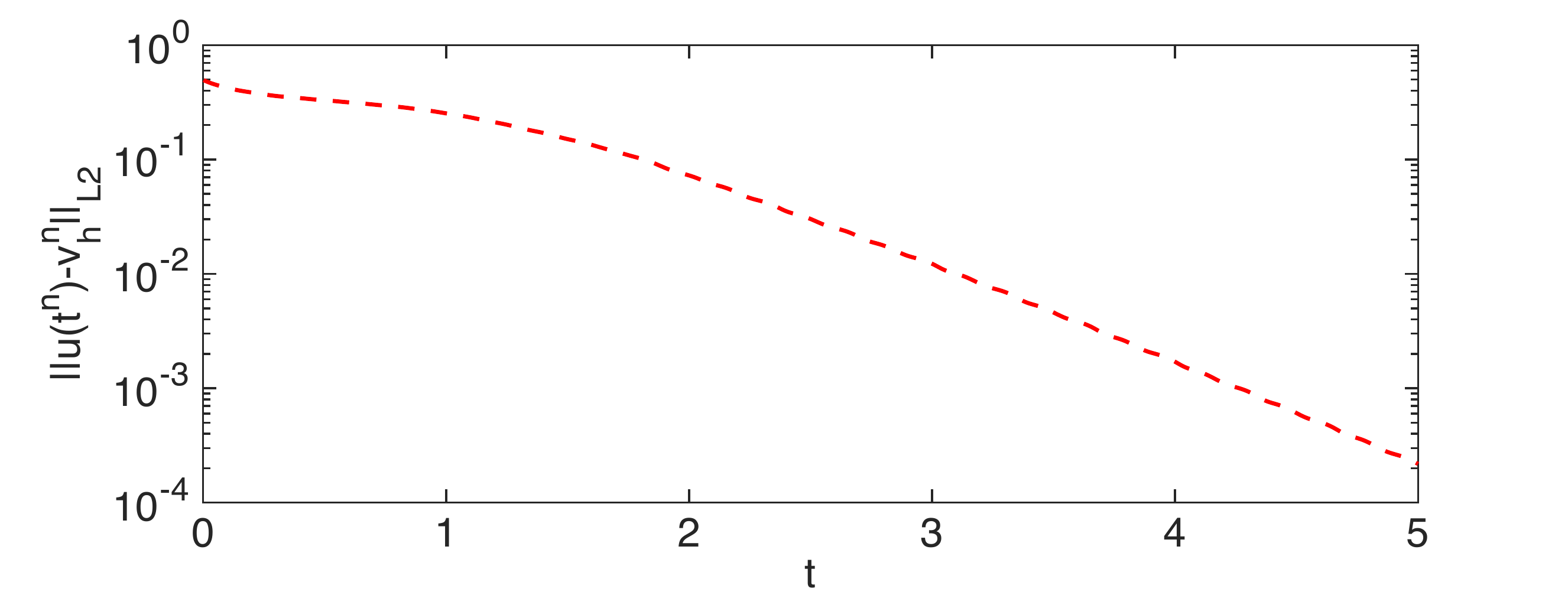}
	\includegraphics[width = .7\textwidth]{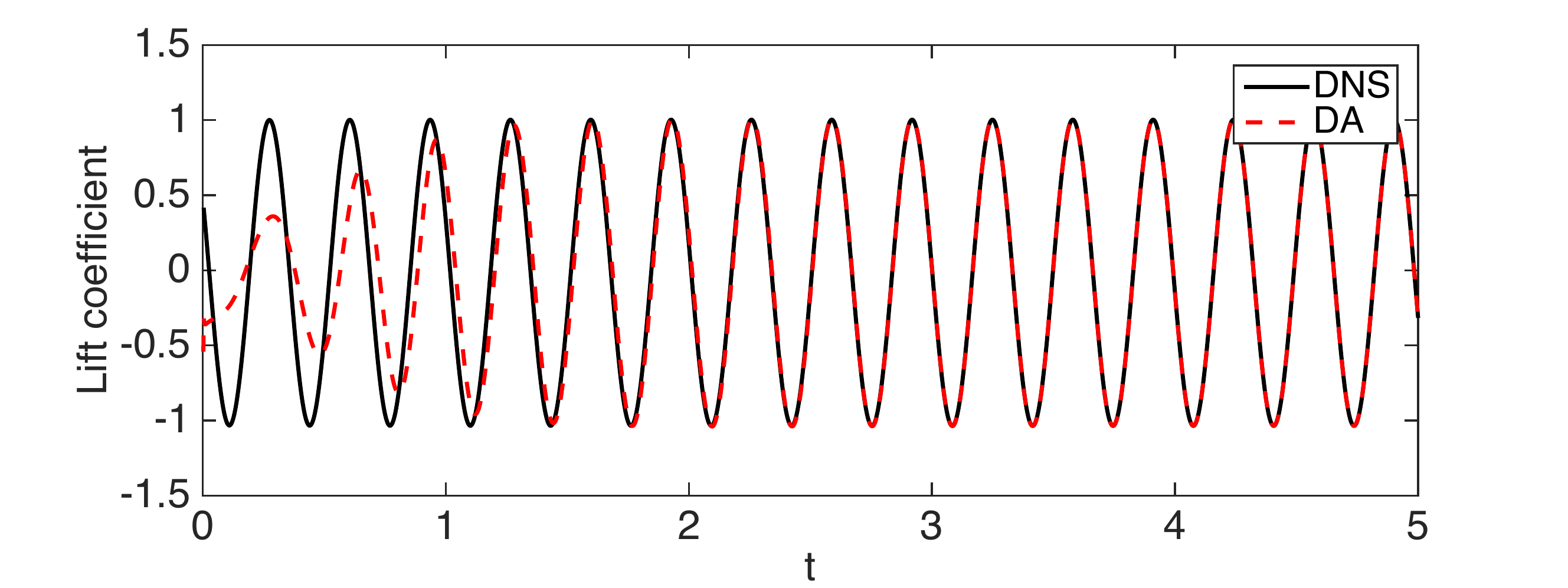}	
	\includegraphics[width = .7\textwidth]{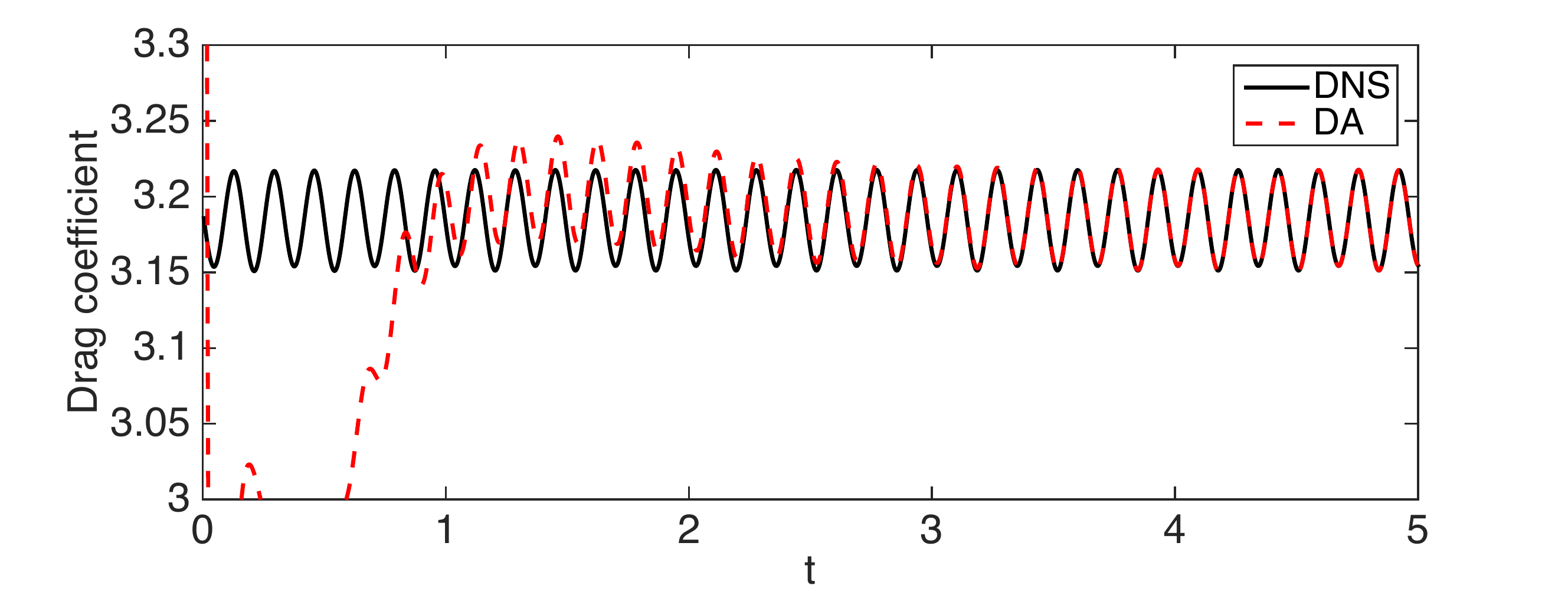}	
	\caption{\label{difference} Shown above is the difference between the DA and DNS versus time, as $L^2$ difference (top), difference in lift coefficients (middle), and difference in drag coefficients (bottom). }
\end{figure}

Results are shown in Figures \ref{difference} and \ref{contourcyl}.  In Figure \ref{difference} at the top, we observe 
exponential decay of $\| v_h^n - u^n \|_{L^2}$ in time, as predicted by our theory.  After 5 seconds, the value is near $10^{-4}$ and is continuing to decreases.  Also in this figure we observe the DA lift and drag slowly catch up to and match the DNS lift and drag: for lift, DA and DNS match by about t=2, but for drag it takes almost to t=3 before there are no visual differences in the plot.  The convergence of the DA solution to the DNS solution in time can also be seen in the speed contour plots in Figure \ref{contourcyl}.  Here, at t=0 there is of course a major difference, since the DA simulation starts at 0.  The accuracy of DA is seen to increase by t=0.5 and further by t=1, and finally by t=2 there is only very slight differences observable between DA and DNS plots.  By t=5, there is no visual difference between DA and DNS, which we expect since the $L^2$ difference between the solutions at t=5 is seen in Figure \ref{difference} to be near $10^{-4}$.

\begin{figure}[!ht]
\begin{center}
DA (t=0) \hspace{2.4in} DNS (t=0)\\
\includegraphics[width = .48\textwidth, height=.17\textwidth,viewport=50 20 650 230, clip]{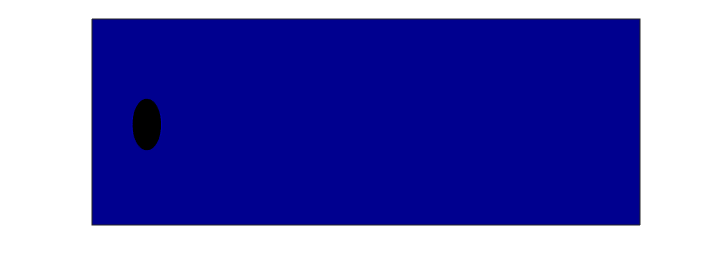}
\includegraphics[width = .48\textwidth, height=.17\textwidth,viewport=50 20 650 230, clip]{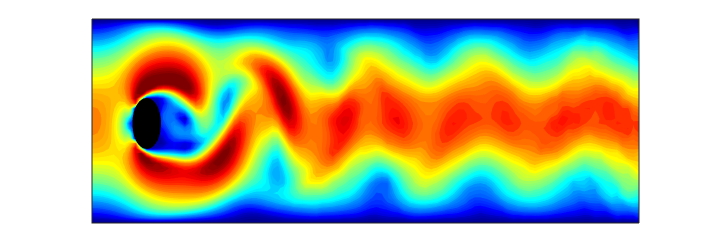}
DA (t=0.5) \hspace{2.4in} DNS (t=0.5)\\
\includegraphics[width = .48\textwidth, height=.17\textwidth,viewport=50 20 650 230, clip]{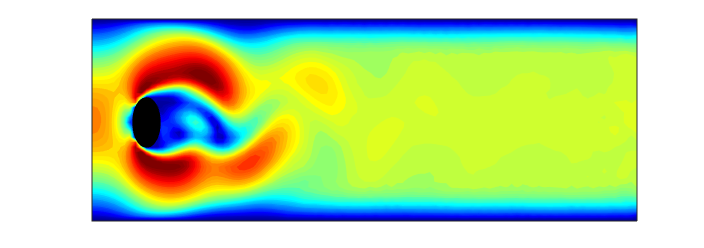}
\includegraphics[width = .48\textwidth, height=.17\textwidth,viewport=50 20 650 230, clip]{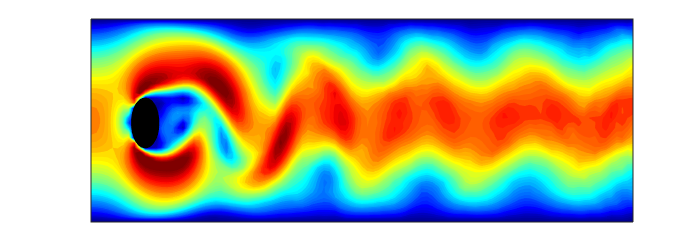}
DA (t=1) \hspace{2.4in} DNS (t=1)\\
\includegraphics[width = .48\textwidth, height=.17\textwidth,viewport=50 20 650 230, clip]{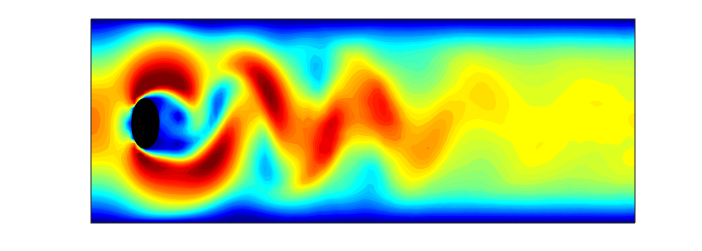}
\includegraphics[width = .48\textwidth, height=.17\textwidth,viewport=50 20 650 230, clip]{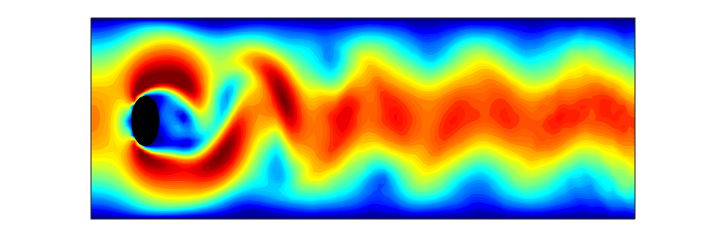}
DA (t=2) \hspace{2.4in} DNS (t=2)\\
\includegraphics[width = .48\textwidth, height=.17\textwidth,viewport=50 20 650 230, clip]{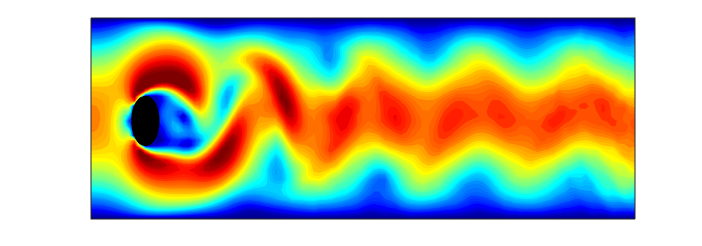}
\includegraphics[width = .48\textwidth, height=.17\textwidth,viewport=50 20 650 230, clip]{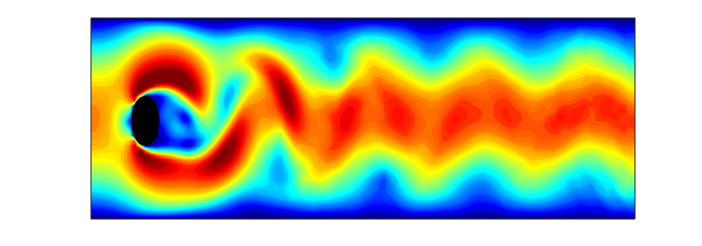}
DA (t=5) \hspace{2.4in} DNS (t=5)\\
\includegraphics[width = .48\textwidth, height=.17\textwidth,viewport=50 20 650 230, clip]{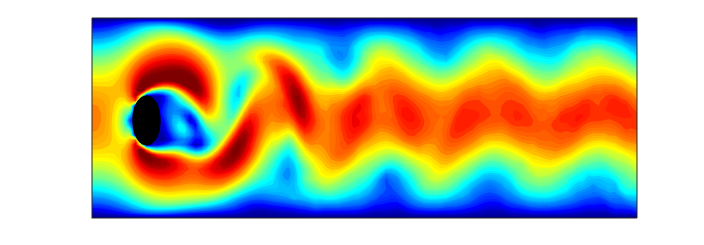}
\includegraphics[width = .48\textwidth, height=.17\textwidth,viewport=50 20 650 230, clip]{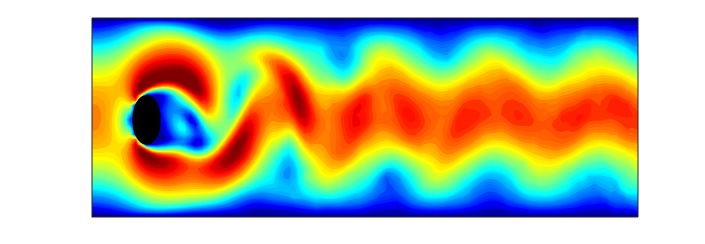}
	\caption{\label{contourcyl} Contour plots of DA and DNS velocity magnitudes at times 0, 0.5, 1, 2, and 5.}
	\end{center}
\end{figure}

\section{Conclusions and Future Directions}

We have analyzed and tested IMEX-finite element schemes for NSE with data assimilation.  Under assumptions that the discretization parameters are sufficiently small, and the NSE solution is sufficiently regular, we proved convergence of the discrete solution to the NSE solution.  Under the assumption of global well-posedness of the NSE solution, our result proves long-time accuracy of the discrete solution.  Several numerical tests were given to show the effectiveness of the scheme, and in particular we found that the element choice can make a dramatic difference in accuracy.  Future directions include considering this approach for related coupled systems, and also to consider long time accuracy in higher order norms.

\end{document}